\documentclass[11pt,a4paper]{amsart}

\usepackage{latexsym}

\usepackage{amssymb}
\usepackage{amsthm}
\usepackage{amsmath}
\usepackage{amstext}
\usepackage{amscd}
\usepackage{color}

\numberwithin{equation}{section}

\theoremstyle{plain}
\newtheorem{thm}{Theorem}[section]
\newtheorem{prop}[thm]{Proposition}
\newtheorem{cor}[thm]{Corollary}
\newtheorem{lem}[thm]{Lemma}

\newtheorem{theorem*}{Theorem}[]

\theoremstyle{definition}

\newtheorem{example}[thm]{Example}

\theoremstyle{remark}
\newtheorem{rem}[thm]{Remark}

\newcommand{\C}{\mathbb{C}}
\newcommand{\N}{\mathbb{N}}
\newcommand{\R}{\mathbb{R}}

\newcommand{\Z}{\mathbb{Z}}

\newcommand{\Sch}{\mathbf{Sch} (\R)}

\newcommand{\HCat}{Ho\, \mathcal{C}}

\newcommand{\proj}{\mathbb{P}}

\newcommand{\inv}{^{-1}}

\DeclareMathOperator{\im}{im}
\DeclareMathOperator{\id}{id}

\DeclareMathOperator{\Ker}{Ker}

\newcommand{\Xb}{\overline X}
\newcommand{\Yb}{\overline Y}
\newcommand{\Zb}{\overline Z}
\newcommand{\fb}{\overline f}

\newcommand{\I}{\mathcal I}
\newcommand{\w}{\widetilde}

\DeclareMathOperator{\simple}{\mathbf s}

\date{August 21, 2012}

\title[Weight filtration II]{The weight filtration for real algebraic varieties II: Classical Homology}

\author[C. McCrory and A. Parusi\' nski]{ Clint McCrory and Adam Parusi\'nski }

\address {Mathematics Department, University of Georgia, Athens GA
30602, USA}

\email{clint@math.uga.edu}

\address {D\'epartement de Math\'ematiques,
Universit\'e de Nice - Sophia Antipolis,
Parc Valrose,
06108 Nice Cedex 02,
France}
   
\email{adam.parusinski@unice.fr}

\begin{document}

 \begin{abstract}
We associate to each real algebraic variety a filtered chain complex, the weight complex, which is well-defined up to filtered quasi-isomorphism, and which induces on classical (compactly supported) homology with $\Z_2$ coefficients an analog of the weight filtration for complex algebraic varieties. This complements our previous definition of the weight filtration of Borel-Moore homology.
 \end{abstract}
 


\maketitle

We define the weight filtration of the homology  of a real algebraic variety by first addressing the case of smooth noncompact varieties. As in Deligne's definition \cite{deligne2} of the weight filtration for complex varieties, given a smooth variety $X$ we consider a \emph{good compactification}, a smooth compactification $\Xb$ of $X$ such that $D = \Xb\setminus X$ is a divisor with normal crossings. Whereas Deligne's construction can be interpreted in terms of the action of a torus $(S^1)^N$, we use the action of a discrete torus $(S^0)^N$ to define a filtration of the chains of a semialgebraic compactification of $X$ associated to the divisor $D$. The resulting filtered chain complex is functorial for pairs $(\Xb,X)$ as above, and it behaves nicely for a blowup with a smooth center that has normal crossings with $D$.

We apply a result of Guill\' en and Navarro Aznar \cite[Theorem (2.3.6)]{navarro} to show that our filtered complex is independent of the good compactification of $X$ (up to quasi-isomorphism) and to extend our definition to a functorial filtered complex, the \emph{weight complex}, that is defined for all varieties and enjoys a generalized blowup property (Theorem \ref{navarro}). For compact varieties the weight complex agrees with our previous definition  \cite{weight} for Borel-Moore homology.

We work with homology rather than cohomology to take advantage of the topology of semialgebraic chains \cite[Appendix]{weight}. We denote by $H_k(X)$ the $k$th classical homology group of $X$, with compact supports and coefficients in $\Z_2$, the integers modulo 2. The vector space $H_k(X)$ is dual to $H^k(X)$, the classical $k$th cohomology group with closed supports. On the other hand, let  $H^{BM}_k(X)$ denote the $k$th Borel-Moore homology group of $X$ (\emph{i.e.} homology with closed supports) with coefficients in $\Z_2$. Then $H^{BM}_k(X)$ is dual to $H^k_c(X)$, the $k$th cohomology group with compact supports.

Our work owes much to the foundational paper \cite{navarro} of Guill\' en and Navarro Aznar. In particular we have been influenced by the viewpoint of section 5 of that paper, on the theory of motives. Using Guill\' en and Navarro Aznar's extension theorems, Totaro \cite{totaro} observed that there is a functorial weight filtration for the cohomology with compact supports of a real analytic variety with a given compactification. In \cite{weight} we developed this theory in detail for real algebraic varieties, working with Borel-Moore homology. Our task was simplified by the strong additivity property of Borel-Moore homology (or compactly supported cohomology) \cite[Theorem 1.1]{weight}. For classical homology or cohomology one does not have such an additivity property, and so the present construction of the weight filtration is more involved.

In section \ref{weightsmoothsection} below we define the weight filtration of a smooth, possibly noncompact, variety $X$, in terms of a good compactification $\Xb$ with divisor $D$ at infinity. First we define a semialgebraic compactification $X'$, the \emph{corner compactification} of $X$, with $X'$ contained in a principal bundle over $\Xb$ with group a discrete torus $\{1,-1\}^N$. We use the action of this group to define the \emph{corner filtration} of the semialgebraic chain group of $X'$. The filtered \emph{weight complex} is obtained from the corner filtration by an algebraic construction, the \emph{Deligne shift}. In section \ref{CechGysinsection} we analyze the relation of the weight complex to the homological \emph {Gysin complex} of the divisor $D$. Section \ref{functorialitysection} contains the proof of the crucial fact that the weight complex is functorial for pairs $(\Xb,X)$, together with an analysis of the functoriality of the Gysin complex. 

Sections \ref{BlowupSquaresection}, \ref{Blowuptransversesection}, and \ref{Blowupcontainedsection} treat the blowup properties of the weight complex of a smooth variety.  A key role is played by the Gysin complex. For example, in section \ref{Blowupcontainedsection} we use the fact that a homomorphism of weight complexes is a filtered quasi-isomorphism if and only if it induces an isomorphism of the homology of the corresponding Gysin complexes.

In section \ref{weightsingularsection} we use the theorems of Guill\' en and Navarro Aznar to extend the definition of the weight complex to singular varieties, and we describe some elementary examples. The appendix (section \ref{appendix}) is devoted to a canonical filtration of the $\Z_2$ group algebra of a discrete torus group. This is in effect a local version of the weight filtration.

By a \emph{real algebraic variety} we mean an affine real algebraic variety in the sense of Bochnak-Coste-Roy \cite{BCR}:  
a topological space with a sheaf of real-valued functions  isomorphic to a real algebraic set $X\subset \R^N$ with the Zariski topology and the structure sheaf of regular functions.  A \emph{regular function on $X$}  is the restriction of a rational function on $\R^N$ that is 
everywhere defined on $X$.  By a \emph{regular mapping} we mean a regular mapping in the sense  of Bochnak-Coste-Roy \cite{BCR}.  

For instance,  the set of real points of a reduced projective scheme over $\R$,  with the sheaf of regular functions, is an affine real algebraic variety in this sense.  
This follows from the fact that  real projective space is isomorphic, as a real algebraic variety, to a subvariety 
of an affine space  \cite[Theorem 3.4.4] {BCR}.  
We also adopt from \cite{BCR} the notion of an algebraic vector bundle.  We recall that such a bundle is, by definition, a subbundle of a trivial vector bundle, and hence it is the pullback of the universal vector bundle on the Grassmannian,  and its  fibers  are generated by global regular sections \cite[Chapter 12]{BCR}.

 By a \emph{smooth real algebraic variety} we mean a nonsingular  affine real algebraic variety.


\section{The weight filtration of a smooth variety}\label{weightsmoothsection}

In this section we define the weight filtration of the classical homology of a smooth variety $X$. We use a smooth compactification $\Xb$ with a normal crossing divisor at infinity to define a semialgebraic compactification $X'$ of $X$ and a surjective map $\pi:X'\to \Xb$ with finite fibers. This map is used to define the weight filtration of the semialgebraic chain complex of $X'$ with $\Z_2$ coefficients. Thus we obtain the weight filtration of the homology of $X'$, which is canonically isomorphic to the homology of $X$. We will prove in Section \ref{weightsingularsection} that this filtration of $H_*(X)$ does not depend on the choice of compactification $\Xb$.

\subsection{The corner compactification}
Let $M$ be a compact smooth   real algebraic variety and let $D\subset M$ be a smooth divisor. Associated to $D$ there is an algebraic line bundle $L$ over $M$ that has a section $s$ such that $D$ is the variety of zeroes of $s$. Let $S(L)$ be the space of oriented directions in the fibers of $L$.  
 It can be given the structure of a real algebraic variety as follows.  By \cite[Remark 12.2.5]{BCR}, $L$ is isomorphic to an algebraic subbundle of the trivial bundle $M\times \R^N$.  Denote by $\Psi:L\to \R^N$ the regular map  defined by this isomorphism.   The scalar product on $\R^N$ defines  a regular metric on $L$.  We identify  $S(L)$ with   the unit zero-sphere bundle of $L$; that is, with the real  algebraic variety $\Psi \inv (S^{N-1})$.   This structure is uniquely defined.  Indeed,  the standard projection $L\setminus M \to \Psi \inv (S^{N-1})$ is a regular map, and therefore 
two such unit sphere bundles are biregularly isomorphic.  Finally, $L$ is the pullback of the universal 
line bundle on $\proj^{N-1}$ under the regular map $M \to \proj^{N-1}$ induced by $\Psi$. 

Thus   $S(L)$ is a smooth real algebraic variety,  and the projection $\pi_L:S(L)\to M$ is an algebraic double covering. 
 Now the subvariety $\pi_L\inv D$ of $S(L)$ is the zero set of the  regular function $\varphi:S(L)\to \R$ defined by $\varphi(x,\ell)\cdot \ell = s(x)$, where $x\in M$ and $\ell$ is a unit vector in the fiber $L_x=\pi_L\inv(x)$. Note that the generator $\tau$ of the group of covering transformations of $S(L)$ changes the sign of $\varphi$, for $\varphi(\tau(x,\ell)) = \varphi(x,-\ell) = -\varphi(x,\ell)$.

Let $X$ be a smooth $n$-dimensional variety, and let $\Xb$ be a \emph{good compactification} of $X$ 
\cite[p.~89]{peterssteenbrink}: $\Xb$ is a compact smooth variety containing $X$, 
and $D = \Xb\setminus X$ is a divisor with simple normal crossings.  
Thus $D$ is a finite union of  smooth codimension one subvarieties $D_i$ of $\Xb$,
\begin{equation}\label{ncd}
D = \bigcup_{i\in I} D_i\ ,
\end{equation}
and the divisors $D_i$ meet transversely. Note that we do not assume that the divisors $D_i$ are irreducible.

For $i \in I$, let $L_i$ be the line bundle on $\Xb$ associated to $D_i$ and let $s_i$ be a section of $L_i$ that defines the divisor $D_i$. Let $\w \pi:\w X\to \Xb$ be the covering of degree $2^{|I|}$ defined as the fiber product of the double covers $\pi_{L_i}:S(L_i)\to \Xb$, and let $\w\varphi_i:\w X\to \R$ be the pullback of the function $\varphi_i:S(L_i)\to \R$ corresponding to the section $s_i$, so that the variety $\w\pi\inv D_i$ is the zero space of $\w\varphi_i$. The \emph{corner compactification} of $X$ associated to the good compactification $(\Xb,D)$ is the semialgebraic set $X'\subset \w X$ defined by
\begin{equation}\label{X'}
X'= \text{Closure}\{\w x\in \w X\ |\ \w\varphi_i(\w x)>0,\ i \in I \}.
\end{equation}
In the terminology of \cite[\S 3.2]{parusinski}, $X'$ is the variety $\Xb$ \emph{cut along the divisor} $D$. Let $\pi:X'\to \Xb$ be the restriction of the covering map $\w\pi:\w X\to \Xb$. 

Let $T$ be the group of covering transformations of the covering space $\w\pi:\w X\to \Xb$, with $\w\tau_i\in T$ the pullback of the nontrivial covering transformation $\tau_i$ of the double cover  $\pi_{L_i}:S(L_i)\to \Xb$. There is a canonical isomorphism $\theta:T\to G$, where $G$ is the multiplicative group of functions $g:I\to \{1,-1\}$, given by $\theta(\w\tau_i) = g_i$, with $g_i(i) = -1$  and $g_i(j) = 1$ for $i\neq j$. To emphasize the role of the group $G$ we prefer to consider
\begin{equation}\label{xi}
\xi=(\w\pi:\w X\to \Xb)
\end{equation} 
as a \emph{principal $G$-bundle} for the group $G=\{1,-1\}^I$ and then 
\begin{equation}\label{actiononsections}
\begin{aligned}
\w\varphi_i(g_i\cdot \w x) &= - \w\varphi_i(\w x),\\
\w\varphi_j(g_i\cdot \w x) &= \w\varphi_j(\w x),\ i\neq j.
\end{aligned}
\end{equation}  
If $U\subset \Xb$ is contractible then $\xi|U$ is trivial, i.e. 
\begin{equation}\label{trivial1}
\begin{aligned}
\w \pi \inv (U) \simeq U\times G
\end{aligned}
\end{equation}
as a  principal $G$-bundle. This isomorphism is uniquely defined by a choice of $x\in U$ and a point $\w x \in \w \pi \inv (x)$, which   we identify via \eqref{trivial1} with $(x, 1)\in U\times G$.   

\begin{prop}\label{localcoordinates}
The  semialgebraic map $\pi:X'\to \Xb$ is surjective. If $x\in \Xb$ let $J(x) = \{i\in I\ |\ x\in D_i\}$ and $G(x) = \{g\in G\ |\ g(i) = 1,\ i\notin J(x)\}$. The fiber $\pi\inv(x)= \{ \w x \in \w \pi \inv (x) \ |\ \w \varphi_i(\w x) >0 \text { for } i \in J(x)\}$.  
Thus $\pi\inv(x)$ is a regular orbit of the action of $G(x)$ on $\w X$;  i.e., a $G(x)$-torsor. Hence the number of points in $\pi\inv(x)$ is $2^{|J(x)|}$.
\end{prop}

\begin{proof}
If $x\in X = \Xb\setminus D$, then $J(x) = \emptyset$ and $G(x)$ is trivial.  
For each $i\in I$ let $\ell_i(x) \in (L_i)_x$ be the unit vector such that $s_i(x)$ is a positive multiple of $\ell_i(x)$; that is, $\varphi_i(x,\ell_i(x))> 0$. Let $\w x = (x,\ell_i(x))_{i\in I}\in \w X$. Then by definition $\w x\in X'$ and $\pi\inv(x) = \{\w x\}$. If $U$ is a contractible, open neighborhood of $x$ in $X$ then the principal bundle $\xi|U$ is trivial, and $\pi\inv(U)$ is a connected component of 
$\w \pi\inv(U)$.  Denote 
\begin{equation}\label{X'+}
X_+'= \{\w x\in \w X\ |\ \w\varphi_i(\w x)>0,\ i \in I \}.
\end{equation}
Thus   $\pi\inv X = X'_+$,  and $\pi$ maps $X_+'$ homeomorphically onto $X$.

Since the divisor $D$ has simple normal crossings \eqref{ncd}, it follows that for every $x\in \Xb$ there is a regular system of  parameters $u_1,\dots, u_n$ for $\Xb$ at $x$, and a semialgebraic open neighborhood $U$ of $x$ such that $(u_1,\dots,u_n)$ is a real analytic, semialgebraic coordinate system on $U$ with $(u_1(x),\dots, u_n(x))=0$, and for each $i\in J(x)$ there is an index $k(i)\in \{1,\dots,n\}$ such that $D_i\cap U$ is the coordinate hyperplane $u_{k(i)}=0$ and 
$$
X\cap U = \{y\in U\ |\ u_{k(i)}(y)\neq 0\ \text{for all}\ i\in J(x) \}.
$$
Then 
\begin{equation}\label{XgU}
\begin{aligned}
X\cap U &= \bigcup_{g\in G(x)} X_g(U),\\
X_g(U) &=  \{y\in U\ |\ g(i)u_{k(i)}(y)> 0\ \text{for all}\ i\in J(x)\},
\end{aligned}
\end{equation}
the set of points of $U$ such that each of the coordinates $u_{k(i)}$ has the sign $g(i)$.

 We say that $(U, (u_1,\dots,u_n))$ is a \emph{good local coordinate system} on $(\Xb, D)$ at $x\in \Xb$ if, moreover, $U$ and all $X_g(U)$ for $g\in G(x)$ are contractible.  Thus $X\cap U$ has exactly $2^{|J(x)|}$ connected components.

Let $y\in X_1(U)$, \emph{i.e.} $u_{k(i)}(y)> 0$ for all $i\in J(x)$, and 
 let $\w U_1$ be the component of $\w\pi\inv U$ containing $\pi\inv X_1(U)$.   We choose the isomorphism in 
 \eqref{trivial1} so that $\w U_1$ corresponds to $U\times \{1\}$.  Then, by \eqref{actiononsections} ,
 \begin{align*}
 \pi\inv X_g(U) = \w \pi\inv X_g(U) \cap g \w U_1.  
 \end{align*}
 In other words  $\pi\inv X_g(U) $ corresponds to  $X_g (U) \times \{g\}$ via the isomorphism  \eqref{trivial1}.  
 In particular, $\pi \inv (x) = \{x\}\times G(x)$ as claimed.    
\end{proof}

As a corollary of the proof we have that every $x'\in X'$ has a neighborhood in $X'$ semialgebraically 
 homeomorphic to a quadrant $\{(u_1,\dots,u_n)\ |\ u_i\geq 0,\ i=1,\dots,m\} \subset \R^n$, where $m=|J(\pi(x'))|$. Thus $X'$ is a semialgebraic manifold with boundary $\partial X'$, and $X'\setminus \partial X' = X'_+$. The inclusion $X\hookrightarrow \Xb$ factors through $\pi$,
$$
\renewcommand{\arraystretch}{1.5}
\begin{array}[c]{ccc} 
& & X' \\
  &\nearrow &\hphantom{\pi}\downarrow{\pi} \\
X & \hookrightarrow & \Xb
\end{array}
$$
and the restriction of $\pi$ to  $X'\setminus \partial X'= X'_+$ is a semialgebraic homeomorphism onto $X$. Thus the inclusion $\lambda:X\to X'$ is a homotopy equivalence, and so $\lambda_*:H_k(X)\to H_k(X')$ is an isomorphism for all $k\geq 0$, where $H_k(X)$ denotes classical homology (with compact supports) with coefficients in $\Z_2$.

\begin{prop}The corner compactification $X'$ of $X$ does not depend  on the choice of sections $s_i$. \end{prop}

\begin{proof} Suppose that for all $i\in I$ we have sections $s_i$ and $\hat s_i$ of $L_i$ defining $D_i$, and these sets of sections define corner compactifications $X'$ and $\widehat X'$, respectively. Suppose there is an index $j\in I$ such that $s_i = \hat s_i$ for all $i\neq j$. If $s_j(x)$ and $\hat s_j(x)$ lie in the same component of the fiber $L_x\setminus \{0\}$ for $x\notin D_j$, then the corresponding functions $\varphi_j$ and $\hat\varphi_j$ have the same sign, so $X'=\widehat X'$.  If $s_j(x)$ and $\hat s_j(x)$ lie in different components of the fiber $L_x\setminus \{0\}$ for $x\notin D_j$, then the corresponding functions $\varphi_j$ and $\hat\varphi_j$ have opposite signs. Thus $g_j(X')=\widehat X'$. \end{proof}

\begin{prop} The corner compactification $X'$ does not depend on the choice of decomposition \eqref{ncd} of the divisor $D$ into smooth subvarieties; that is, two such compactifications are canonically semialgebraically homeomorphic. \end{prop}

\begin{proof} Suppose that the divisor $D_j$ is the union of two nonempty smooth divisors $D_a$ and $D_b$, $D_a\cap D_b = \emptyset$, and we replace $D_j$ with $D_a\cup D_b$ in the decomposition \eqref{ncd}. Then the line bundle $L_j$ equals $L_a\otimes L_b$, and we can take $s_j= s_a\otimes s_b$.   If we choose the metric 
 on $L_j$ to be the product of the metrics on $L_a$ and $L_b$, then $\varphi_j = \varphi_a\cdot \varphi_b$, \emph{i.e.} $\varphi_j(x,\ell_a\otimes \ell_b) = \varphi_a(x,\ell_a) \varphi_b(x,\ell_b)$, and we have a double cover $S(L_a)\times_{\Xb}S(L_b)\to S(L_j)$ given by $((x,\ell_a),(x,\ell_b))\mapsto (x,\ell_a\otimes \ell_b)$. Let $\w X(j)$ be the fiber product of the double covers $S(L_i)$ with $L_j$ replaced by $L_a$ and $L_b$, and let $ X'(j)$ be the resulting corner compactification. Then the double cover $p:\w X(j)\to \w X$ restricts to a semialgebraic homeomorphism $X'(j)\to X'$. To prove this it suffices to show that $p$ restricts to a bijection $X'(j)\setminus\partial X'(j)\to X'\setminus\partial X'$. Suppose that $\w x = (x,\ell_i)_{i\in I}\in X'\setminus\partial X'$. Then $\w \varphi_i(x,\ell_i)> 0$ for all $i\in I$, and in particular $\varphi(x,\ell_j)>0$. Let $\ell_j = \ell_a\otimes \ell_b$, so that $\varphi_j(x,\ell_j) = \varphi_a(x,\ell_a) \varphi_b(x,\ell_b)>0$. Now $p\inv(\w x)= \{\w y,\w z\}$, where $\w y$ is obtained from $\w x$ by replacing $(x,\ell_j)$ with $((x,\ell_a),(x,\ell_b))$ and $\w z$ is obtained from $\w x$ by replacing $(x,\ell_j)$ with $((x,-\ell_a),(x,-\ell_b))$. Thus if $\varphi_a(x,\ell_a)>0$ we have $\w y\in X'(j)$ and $\w z\notin X'(j)$, and if $\varphi_a(x,\ell_a)<0$ we have $\w y\notin X'(j)$ and $\w z\in X'(j)$. \end{proof}


\subsection{The corner filtration}
We will use the map $\pi:X'\to \Xb$ and the action of the group $G$ on $\w X$ to define a filtration of the semialgebraic chain complex $C_*(X')$ 
of the corner compactification $X'$. Given the decomposition \eqref{ncd} of $D$, for  $J\subset I$ we define 
\begin{equation}\label{DJ}
\begin{aligned}
D_J &= \bigcap_{i\in J} D_i\ ,\   D_\emptyset = \Xb,\\
\mathring D_J &=  D_J \setminus \bigcup_{i\notin J} D_i\ ,\ \mathring D_\emptyset = X .
\end{aligned}
\end{equation}
Then $\{\mathring D_J \}_{J\subset I}$ is a stratification of $\Xb$. This is a local condition that follows from the fact that every $x\in \Xb$ is contained in a good coordinate system $(U, (u_1,\dots,u_n))$, with 
$$
D\cap U = \{y\in U\ |\ u_{k(i)} = 0\ \text{for some}\ i\in J(x)\}.
$$ 
 In these coordinates, for $J\subset J(x)$ we have 
\begin{equation*}
\begin{aligned}
D_J\cap U &= \{y\in U\ |\ u_{k(i)}(y) = 0,\ i\in J\},\\
\mathring D_J\cap U &= (D_J\cap U) \cap \{y\in U\ |\ u_{k(i)}(y) \neq 0,\ i\in J(x) \setminus J\}.
\end{aligned}
\end{equation*}
Similarly we  can stratify $X'$ by taking $\pi \inv \mathring D_J$ as strata, and $\pi \inv D_J$ is the closure in $X'$ of the stratum $\pi \inv \mathring D_J$.  To prove these assertions, let $(U, (u_1,\dots,u_n))$ be a good coordinate system as above, and let $\w U_1$ be the component of $\w\pi\inv U$ containing $\pi\inv X_1(U)$. 
For $g\in G(x)$ let 
\begin{equation}\label{XbargU}\Xb_g(U) = \{y\in U\ |\ g(i) u_{k(i)}(y)\geq 0,\ i\in J(x)\},\end{equation}
 the closure in $U$ of $X_g(U)$.  
We have that $\w\pi$ maps $(g \w U_1, \pi\inv\Xb_g(U))$ homeomorphically onto $(U, \Xb_g(U))$, and $\pi\inv U$ is the disjoint union of the sets $\pi\inv\Xb_g(U)$.  Clearly $\{\mathring D_J\cap \Xb_g(U)\}_{J\subset J_0}$ is a stratification of $\Xb_g(U)$, and the closure of $\mathring D_J\cap \Xb_g(U)$ in $\Xb_g(U)$ is $ D_J\cap \Xb_g(U)$.

\vskip.1in

Now for each $J\subset I$ such that $D_J \ne \emptyset$,  let $G(J)=\{g\in G\ |\ g(i) = 1,\ i\notin J\}$.   
Then $G(J)$ is isomorphic to $\{1,-1\}^{|J|}$  and for each $x\in \mathring D_J$ we have $G(x) = G(J)$.   Thus the action of $G(J)$ on $\w X$ preserves $\pi\inv D_J$.  
 Consider the inclusion $C_* (\pi\inv D_J)  \to  C_*(X')$. 
 We denote by $F^{J} C_*(X')$  the image in $C_*(X')$ of the subcomplex of $C_* (\pi\inv D_J) $ of $G(J)$-invariant chains.
Then $F^{J} C_*(X')$ is a subcomplex of $C_*(X')$.   For $p\geq0$ let $ F^{p}C_*(X')$ be the subcomplex  of $C_*(X')$ generated by the  $F^{J} C_*(X')$ with $|J|=p$. 

If $J \subset K$ then $D_J \supset D_K$ and $G(J) \subset G(K)$.  Therefore $F^{J} C_*(X') 
\supset F^{K} C_*(X')$.  So for all $p\geq 0$ we have $F^{p+1}C_*(X')\subset F^pC_*(X')$. We obtain a filtration  
\begin{equation}\label{cornercomplex}
C_*(X') = F^0C_*(X')\supset F^1C_*(X')\supset F^2C_*(X')\supset \cdots ,
\end{equation}
with $F^{n-k+1}C_k(X') = 0$ for all $k\geq 0$, where $n = \dim X$.  We call this filtered complex the \emph{corner complex} of the good compactification $(\Xb,D)$ of $X$. 

The \emph{corner spectral sequence} $\widehat E^r_{p,q}$ is the spectral sequence associated to the  
increasing filtration $\widehat F_*$ obtained by setting $\widehat F_{-p} = F^p$,
\begin{equation}\label{cornerss}
\cdots\subset \widehat F_{-2}C_*(X')\subset \widehat F_{-1}C_*(X')\subset \widehat F_0C_*(X') = C_*(X').
\end{equation}
 This is a second quadrant spectral sequence: If $\widehat E^r_{p,q} \neq 0$ then $(p,q)$ lies in the closed triangle with vertices $(0,0)$, $(0,n)$, $(-n,n)$, $n = \dim X$. The corner spectral sequence converges to the homology of the corner compactification $X'$,
 $$
 \widehat E^r_{p,q}\implies H_{p+q}(X') .
 $$
 
It will be useful to describe the corner filtration on the level of semialgebraic sets, using the definition of semialgebraic chains given in the appendix of \cite{weight}. If $\Gamma$ is a closed $k$-dimensional semialgebraic subset of a semialgebraic set $X$, then $c=[\Gamma]\in C_k(X)$ is the semialgebraic chain represented by $\Gamma$. 

The vector subspace $F^pC_k(X')$ is generated by the subspaces $F^JC_k(X')$ for $J\subset I$, and our definition implies that $c\in F^JC^k(X')$ if and only if $c=[\Gamma]$, where $\Gamma\subset X'$ and $G(J)\Gamma = \Gamma$.

 \vskip.1in
Next we give an alternative  description of the corner filtration. 
For each $J\subset I$ such that $D_J \ne \emptyset$  consider the corner compactification $D'_J$ of  $\mathring D_J$ associated to  the good compactification $D_J$ of $\mathring D_J$ with divisor $ \bigcup_{i\notin J}( D_i\cap D_J)$, and let $\pi_J:D_J'\to D_J$ be the projection.

\begin{prop}\label{factor}
The projection $\pi\inv D_J \to D_J$ factors through $D'_J$.  The induced map  $\rho_J: \pi\inv D_J  \to  D'_J$ is a  principal $G(J)$-bundle, and hence it is a  covering space of degree $2^{|J|}$.  
 \end{prop}
 \begin{proof}
Let $\xi=(\w\pi:\w X\to \Xb)$ be the principal $G$-bundle associated to the good compactification $\Xb$ of $X$ \eqref{xi}. The restriction $\xi|D_J = (\w\pi:\w\pi\inv D_J\to D_J)$ is a principal $G$-bundle. Let 
$\xi_J=(\w\pi_J:\w D_J\to D_J)$ be the principal $G(I \setminus J)$-bundle associated to the good compactification $D_J$ of $\mathring D_J$. Let 
$\zeta_J=(\w\rho_J:\w\pi\inv D_J\to \w D_J)$ be the principal $G(J)$-bundle such that $\w\rho_J$ is the quotient map of the action of $G(J)$ on $\w\pi\inv D_J$. On $\w\pi\inv D_J$ we have $\w\pi = \w\pi_J\circ \w\rho_J$.

Now $(\w\rho_J)\inv D_J'=\pi\inv D_J$. If $\rho_J =\w\rho_J|\pi\inv D_J$  then on $\pi\inv D_J$ we have  $\pi= \pi_J\circ \rho_J$, and $\zeta_J|D_J'=(\rho_J: \pi\inv D_J\to D_J')$ is a principal $G(J)$-bundle.
 \end{proof}
 
 \begin{cor}\label{trivial}
There is a finite semialgebraic open cover $\mathcal U_J$ of $ \mathring D_J$ such that over each $U\in \mathcal U_J$ the  projection $\pi\inv U \to U$ is a trivial $G(J)$-bundle, i.e. $\pi\inv U = U\times G(J)$.  
 \end{cor}
 
 \begin{proof}
 This is true for  $\rho_J: \pi\inv D_J  \to  D'_J$ because  $D'_J$ is compact.  Now  
 $\pi_J: \pi_J\inv (\mathring D_J)\to \mathring D_J$ is an isomorphism and hence $\mathring D_J$ 
 can be identified with  $\pi_J\inv (\mathring D_J) \subset D'_J$.  
 Thus it suffices to restrict to   $\mathring D_J$ the corresponding open cover of $D'_J$.
 \end{proof}
 
Associated to the principal bundle $\rho_J: \pi\inv D_J\to D_J'$ of Proposition \ref{factor}, we have the \emph{inverse image map} $\rho_J^* : C_* (D'_J) \rightarrow C_* (\pi\inv D_J)$ defined by $\rho_J^*([\Gamma])= [\rho_J\inv\Gamma]$.
The function $\rho_J^*$ commutes with the boundary map, and so $\rho_J ^*$ is an injective morphism of complexes. (The map $\rho_J^*$ is the chain-level \emph{transfer homomorphism} of the covering map $\rho_J$.) Let $i_J:C_* (\pi\inv D_J)  \to  C_*(X')$ be the inclusion.  Then  $F^{J} C_*(X')$ is  the image in $C_*(X')$ of the composition 
$\eta_J= i_J\circ \rho_J^*$,
\begin{equation}\label{FJ}
\eta_J:C_* (D'_J) \stackrel{\displaystyle \rho_J^*}\longrightarrow C_* (\pi\inv D_J)  \stackrel{\displaystyle i_J }\longrightarrow  C_*(X') ,
\end{equation}
and $\eta_J$ is an isomorphism of the complexes $C_* (D'_J)$ and $F^{J} C_*(X')$.  
Thus $c\in F^{J} C_*(X')$ if and only if $c=[\Gamma]$ for $\Gamma\subset \pi\inv D_J$ with $\Gamma=\rho_J\inv B$, where $B\subset D_J'$.

From Corollary \ref{trivial} we obtain the following useful local characterization of the corner filtration. The vector space $F^{J} C_*(X')$ is generated by the chains $c\in C_*(X')$ such that $c=[\Gamma]$ with  $\Gamma\subset \pi\inv D_{J'}$ for $J'\supset J$ (so $D_{J'}\subset D_J$), and $\Gamma=\operatorname{Closure} \mathring\Gamma$, where 
$\mathring \Gamma \subset \pi \inv \mathring B$, with $\mathring B \subset \mathring D_{J'}$, $\pi\inv \mathring B =\mathring  B \times G(J')$ (\emph{i.e.} $\pi\inv \mathring B\to \mathring B$ is a trivial $G(J')$-bundle)    and 
 $\mathring\Gamma=\mathring B\times gG(J)$ for some $g\in G(J')$. In other words, $\mathring\Gamma$ is an orbit of the action of $G(J)$ on $ \pi\inv \mathring B$.

 Let $\mathring B \subset \mathring D_{J'}$ be a semialgebraic set such that $\pi\inv \mathring B = \mathring B \times G(J')$, let $\dim \mathring B=k$, and let $B$ be the closure of $\mathring B$.  
Then 
\begin{equation}\label{product}
C_k (\pi\inv B) = C_k ( B) \otimes C_0 (G(J')), 
\end{equation} 
where we consider $G(J')$ as a discrete topological space.  In particular, $C_0(G(J')) = \Z_2 [G(J')]$ the $\Z_2$ group algebra of $G(J')$.  
Using this algebra structure we define in the Appendix \eqref{Ifiltration} a filtration $\I^*$ on $ \Z_2 [G(J')]$ and 
hence on   $C_0(G(J')) $.  

 \begin{lem}\label{localtrivialization}
$C_k(\pi\inv B)\cap F^p C_k (X') = C_k ( B) \otimes \I^p C_0 (G(J'))$.  
\end{lem}
\begin{proof}
This follows from Proposition \ref{generating}. By Proposition \ref{translation}, the right hand side does not depend on the choice of isomorphism $\pi\inv\mathring B = \mathring B\times G(J')$.
\end{proof}

 \begin{prop}\label{shortexact}
The homomorphism
$\pi_* : C_*(X') \to  C_*(\Xb)$ induces an exact sequence 
\begin{align*}
0 \to F^1 C_*(X')  \to C_*(X') \to  C_*(\Xb) \to 0 .
\end{align*}  
\end{prop}
\begin{proof}
Fix $\mathring B \subset \mathring D_{J'}$,  $\dim \mathring B =k$, as above.  It suffices to check the exactness 
for $k$-chains over $B$; that is, the exactness of the sequence 
\begin{align*}
0 \to C_k(\pi \inv B)  \cap F^1 C_k (X') \to C_k(\pi \inv B)  \to  C_k(B) \to 0 .
\end{align*}  
This follows  from  Lemma \ref{localtrivialization} and the definition of $\mathcal I^1$ as the kernel of the augmentation map $\epsilon : C_0 (G(J')) \to \Z_2$; see the Appendix \eqref{augmentation}.   
\end{proof}
 
Now we compute the successive quotients of the corner filtration. In Section \ref{CechGysinsection} below we will use the following result to show that the $(\widehat E^1,\hat d^1)$ term of the corner spectral sequence is isomorphic to the Gysin complex of the divisor $D$ (Corollary \ref{E1Gysin}).

 \begin{prop}\label{isotoGysin}
For each $p\geq 0$ there is an isomorphism of chain complexes 
$$
\psi_p: \bigoplus_{|J| = p} C_*(D_J)\stackrel{\displaystyle \approx}\longrightarrow \frac {F^{p} C_*(X') }{F^{p+1} C_*(X')}\  .
 $$
\end{prop}

\begin{proof} 
First we consider the case $p=0$.  By Proposition \ref{shortexact},   
$\pi_* : C_*(X') \to  C_*(\Xb)$ induces  an isomorphism  $\psi  :  C_*(\Xb) \to 
 C_*(X') / F^{1} C_*(X')  $ that can be described geometrically  as follows.    Given a chain $b\in C_k(\Xb)$ represented by the set $B$, then $c= \psi (b)$ is represented modulo 
$ F^1 C_*(X')$ by the closure $\Gamma$ of the image of any semialgebraic (not necessarily continuous) section of $\pi$ over $B$.  

Similarly we construct $\psi_p$ for any  $p\ge 0$.  
Let  $b=[B]\in C_k(D_J)$, $p=|J|$, and let  $\Gamma\subset X'$ be the closure of the image of any semialgebraic section of $\pi$ over $B$.  
Then we define $\psi_p(b) =c\in  F^pC_k(X')\pmod{F^{p+1}C_k(X')}$, where $c= [G(J)\Gamma]$.   
 We have to show that $\psi_p$ is well-defined, injective, surjective, and that it commutes with the boundary.  
 For this we use the characterization of the corner filtration $F^*$ given in   Lemma \ref{localtrivialization} and the Appendix.  
 
Let $b=[B]$, with $B\subset D_J$, and let $\Gamma'$, $\Gamma''\subset X'$ be the closures of the images of semialgebraic sections of $\pi$ over $B$.   By Corollary \ref{trivial}, after a subdivision of  $B$ 
we may suppose that $B$ is the closure of $\mathring B$, where
$\mathring B\subset \mathring D_{J'}$, $J\subset J'$, and that $\pi\inv \mathring B$ is isomorhpic to  $\mathring B\times G(J')$ as a principal $G(J')$-bundle.  Moreover, by a choice of this isomorphism, and another subdivision of 
 $B$  if necessary, we may also suppose that $\Gamma'$ is the closure of $\mathring\Gamma'$, and $\Gamma''$ is the closure of $\mathring\Gamma''$, where $\mathring\Gamma' = \mathring B\times  \{1\}$ and $\mathring\Gamma'' = \mathring B\times  \{g\}$.    If $g\in G(J)$ then $G(J) \mathring\Gamma' = G(J) \mathring\Gamma''$ so suppose $g\not \in G(J)$.  Let $G'$ be the subgroup of $G(J')$ generated by $G(J)$ and $g$.  Then    $ [G(J) \Gamma'] - [G(J)\Gamma''] = [G' \Gamma']   \in F^{p+1}C_k(X')$, by Lemma \ref{localtrivialization} and Lemma \ref{rank}.    This shows that $\psi_p$ is well-defined.  
 
We now show the injectivity of $\psi_p$.  By a  reduction as  in the previous argument it suffices to show the following 
 claim. Let $b= [B]\in C_k(D_{J'})$, $|J'|\geq p$, where $B$ is the closure of $\mathring B$, with $\mathring B\subset \mathring D_{J'}$ and $\pi\inv \mathring B = \mathring B\times G(J')$. Let $\Gamma$ be the closure of the image of a semialgebraic section of $\pi$ over $B$. We claim that  if 
 $$
 \sum_{J\subset J', |J| = p}  a_J [G(J)\Gamma]\in F^{p+1} C_k(X'), \quad a_J\in \Z_2, 
 $$
then $a_J=0$ for all $J$. Now,  in terms of the isomorphism \eqref{product}, 
$$
 \sum_{J\subset J', |J| = p}  a_J [G(J)\Gamma]=[\Gamma] \otimes \left(\sum_{J\subset J', |J| = p}  a_J [G(J)] \right),
$$
so the claim follows from Corollary \ref{nchoosep}.
  
We now reinterpret the restriction of $\psi_p$ to  $C_*(D_J)$.  We denote this restriction by $\psi_J$.  Denote by $\psi_{J,0}$ the isomorphism of complexes from $C_* (D'_J)/ F^1 C_* (D'_J)$ to $C_* (D_J)$.  Then $\psi_J = \eta'_J \circ (\psi_{J,0})\inv $, where $\eta '_J$ equals $\eta_J$  modulo $F^{p+1}C_*(X')$ (see \eqref{FJ}). It follows that $\psi_J$ commutes with the boundary. Since  the images of all $\eta_J$, $|J|=p$, generate $F^pC_*(X')$, the images of $\psi_J$, $|J|=p$, generate $F^pC_*(X')/ F^{p+1}C_*(X')$. This shows $\psi_p$ is surjective.   
\end{proof}
  

\subsection{The weight filtration}\label{weightfiltrationdef}

The \emph{weight filtration} $\mathcal W_*$ of $C_*(X')$ is defined by
\begin{equation}\label{weightfiltration}
\begin{aligned}
\mathcal W_pC_k(X') &= \Ker[\partial: \widehat F_{p+k}C_k(X')\to C_{k-1}(X')/\widehat F_{p+k-1}C_{k-1}(X')] \\
&= (\operatorname{Dec} \widehat F)_pC_k(X'),
\end{aligned}
\end{equation}
where $\operatorname{Dec} \widehat F_*$ is the  \emph{Deligne shift}  of the filtration $\widehat F_*$ \eqref{cornerss} \cite[(1.3.3]{deligne2}, \cite[A.50]{peterssteenbrink}. Thus the weight filtration $\mathcal W_*$ runs
\begin{equation}\label{weightcomplex}
0 = \mathcal W_{-n-1}C_k(X')\subset \cdots\subset \mathcal W_{-k-1}C_k(X')\subset \mathcal W_{-k}C_k(X') = C_k(X').
\end{equation}
We denote this filtered complex by $\mathcal WC_*(X')$. It is the \emph{weight complex} of the good compactification $\Xb$ of $X$.

The \emph{weight spectral sequence} $E^r_{p,q}$  is the spectral sequence associated to the weight complex. It is a second quadrant spectral sequence: If $E^r_{p,q} \neq 0$ then $(p,q)$ lies in the closed triangle with vertices $(0,0)$, $(-n,2n)$, $(-n,n)$, $n = \dim X$. We have
$$
E^r_{p,q} = \widehat E^{r+1}_{2p+q,-p}
$$
for all $r\geq 1$ and all $p, q$. In particular the $E^1$ term of the weight spectral sequence equals the reindexed $\widehat E^2$ term of the corner spectral sequence. The weight spectral sequence converges to the homology of the corner compactification $X'$,
$$
E^r_{p,q}\implies H_{p+q}(X') .
$$

In Section \ref{weightsingularsection} below we will prove that, up to filtered quasi-isomorphism, the weight complex is independent of the good compactification of $X$. Thus the induced filtration on $H_*(X)$ and all the terms of the weight spectral sequence $(E^r_{p,q}, d^r)$ for $r\geq1$ are algebraic invariants of $X$.


\section{The \v Cech and Gysin complexes}\label{CechGysinsection}
We show that the corner complex of a good compactification $(\Xb,D)$ of $X$ is filtered quasi-isomorphic to the cohomology \v Cech complex of the divisor $D$. It follows that the term $(\widehat E^1,\hat d^1)$ of the corner spectral sequence is isomorphic to the Gysin complex of the divisor $D$.

The semialgebraic cohomology groups of a variety $Y$ are dual to the semialgebraic homology groups of $Y$ (coefficients in $\Z_2$). The cohomology groups $H^k(Y)$, $k\geq 0$, are the homology groups of the semialgebraic cochain complex $(C^*(Y),\delta)$, where $C^k(Y) = \hom(C_k(Y),\Z_2)$ and the coboundary map $\delta_k:C^k(Y)\to C^{k+1}(Y)$ is the adjoint of the boundary map $\partial_{k+1}:C_{k+1}(Y)\to C_k(Y)$.


\subsection{The \v Cech complex}\label{cechsection}
Consider the double complex
\begin{equation}\label{cechcomplex}
C^{p,q} = \bigoplus_{|J| = p} C^q(D_J)
\end{equation}
with first differential $\delta': C^{p,q}\to C^{p+1,q}$ the sum of the restriction maps $C^q(D_J)\to C^q(D_{J'})$ ($|J|= p$, $|J'|= p+1$, and $J\subset J'$) \eqref{DJ}, and second differential $\delta'': C^{p,q}\to C^{p,q+1}$ the sum of the coboundary maps $C^q(D_J)\to C^{q+1}(D_J)$. The cohomology \emph{\v Cech complex} of $(\Xb,D)$ is the complex $\check C^l(\Xb,D) = \bigoplus_{p+q = l}C^{p,q}$ with differential $\delta = \delta'+\delta''$ and decreasing filtration $F^p\check C^l(\Xb,D) = \bigoplus_{j \geq p}\bigoplus_{j+q=l} C^{j,q}$. The cohomology spectral sequence $\check E_r^{p,q}$ associated to this filtration \cite[chapter XI, section 8]{maclane} satisfies 
$$
\check E^{p,q}_1 = \bigoplus_{|J| = p} H^q(D_J),
$$
where the differential $\check d_1^{p,q}:\check E_1^{p,q}\to \check E_1^{p+1,q}$ is equal to the sum of the restriction maps $H^q(D_J)\to H^q(D_{J'})$. This spectral sequence converges to the relative cohomology of the pair $(\Xb,D)$,
$$
\check E_1^{p,q}\implies H^{p+q}(\Xb,D).
$$  


\subsection{The Gysin complex}\label{gysinsection}
If $f:M\to N$ is a continuous map of compact manifolds without boundary, the \emph{Gysin homomorphism} $f^*:H_*(N)\to H_*(M)$ is defined as follows. Let $m=\dim M$ and $n=\dim N$, and for all $l\geq 0$ let $\mathcal D_M:H^l(M)\to H_{m-l}(M)$ and $\mathcal D_N:H^l(N)\to H_{n-l}(N)$ be the Poincar\' e duality isomorphisms. For all $k\geq 0$ we let
$$
f^*= \mathcal D_M\circ H^{n-k}(f) \circ \mathcal D_N\inv: H_k(N)\to H_{k+m-n}(M),
$$
where $H^{n-k}(f):H^{n-k}(N)\to H^{n-k}(M)$ is the homomorphism induced by $f$ on cohomology.

The \emph{Gysin complex} $G(\Xb,X)$ of the good compactification $(\Xb,D)$ of $X$ is the chain complex
\begin{align}\label{Gysincomplex}
G_p(\Xb,X) = \bigoplus_{|J| = p}\bigoplus_k H_k(D_J)
\end{align}
with differential $d$, where $d_p:G_p(\Xb,X)\to G_{p+1}(\Xb,X)$ is the sum of the Gysin maps $i_{J,J'}^*: H_k(D_J)\to H_{k-1}(D_{J'})$ for $J\subset J'$ with $|J| = p$ and $|J'| = p+1$, and $i_{J,J'}:D_{J'}\to D_J$ is the inclusion. 

Note that while the \v Cech complex is the filtered complex associated to a double complex, the Gysin complex is a simple complex without filtration.

\begin{prop}\label{cech-gysin}
The $\check E_1$ term of the \v Cech spectral sequence of a good compactification of $X$ is canonically isomorphic to the Gysin complex,
$$
(\check E_1^{p,*}, \check d_1^{p,*})\cong (G_p(\Xb,X), d_p).
$$
\end{prop}
\begin{proof}
The isomorphism $\check E_1^{p,*}\to G_p(\Xb, X)$ is the sum of the Poincar\' e duality isomorphisms $H^q(D_J)\to H_{n-p-q}(D_J)$, where $n=\dim X$.
\end{proof}


\subsection{Poincar\' e-Lefschetz duality}\label{poincarelefschetzsection}

The following isomorphism corresponds to the classical duality isomorphism $H^{n-k}(\Xb,D)\cong H_k(X)$, $n=\dim X$, $k\geq 0$.

\begin{thm}\label{cech-corner} Let $(\Xb,D)$ be a good compactification of $X$, and let $X'$ be the associated corner compactification of $X$. There is a quasi-isomorphism of filtered complexes
$$
\Psi: (\check C^*(\Xb,D), F^*) \to (C_*(X'), F^*)
$$
from the cohomology \v Cech complex to the corner complex. More precisely, there is a chain homomorphism $\Psi = (\Psi_l)$,  $\Psi_l: \check C^l(\Xb,D) \to C_{n-l}(X')$, such that for all $p,l\geq 0$ we have 
$$
\Psi_l( F^p\check C^l(\Xb,D)) \subset F^pC_{n-l}(X'),
$$
and for all $p\geq 0$ the resulting chain homomorphism
$$
\Psi_*: \frac{F^p\check C^*(\Xb,D)}{F^{p+1}\check C^*(\Xb,D)}\to \frac{F^pC_{n-*}(X')}{F^{p+1}C_{n-*}(X')}\ 
$$
induces an isomorphism in homology.
\end{thm}

\begin{cor}\label{E1Gysin}
The $\widehat E_1$ term of the corner spectral sequence of a good compactifcation of $X$ is isomorphic to the Gysin complex of the divisor $D$ at infinity. More precisely,
for every $p, k \geq 0$ there is an isomorphism
$$
\widehat E^1_{-p,k+p} = H_k\left(\frac{F^pC_*(X')}{F^{p+1}C_*(X')}\right)\cong \bigoplus_{|J| = p} H_k(D_J).
$$
Under this isomorphism the differential 
$$
\hat d^1_{-p, k+p}: H_k\left(\frac{F^pC_*(X')}{F^{p+1}C_*(X')}\right)\to 
H_{k-1}\left(\frac{F^{p+1}C_*(X')}{F^{p+2}C_*(X')}\right)
$$ 
corresponds to the sum of the Gysin homomorphisms 
$$
i_{J,J'}^*:H_k(D_J)\to H_{k-1}(D_{J'}),
$$
 where $|J|= p$, $|J'|=  p+1$, and $J\subset J'$.
\end{cor}

\begin{proof}
This is an immediate consequence of  Theorem \ref{cech-corner} and Proposition \ref{cech-gysin}.
\end{proof}

Now we turn to the proof of Theorem \ref{cech-corner}. We construct $\Psi$ as the composition of the three filtered quasi-isomorphisms described in subsections \ref{simplicialcechsection}, \ref{cellulargysinsection}, \ref{cellularpullbacksection} below.


\subsection{The simplicial \v Cech complex}\label{simplicialcechsection}
Let $K$ be a semialgebraic triangulation of $\Xb$ such that for all $J\subset I$ the subvariety $D_J=\bigcap_{i\in J}D_i$ is a subcomplex of $K$. There exists a unique triangulation $K'$ of $X'$ such that the map $\pi:X'\to\Xb$ is simplicial, and such a triangulation $K'$ is semialgebraic. We say that $(K',K)$ is an \emph{adapted triangulation} of $\pi$. 

Let $(K',K)$ be an adapted triangulation of $\pi:X'\to \Xb$. For each $J\subset I$ let $K_J$ be the subcomplex of $K$ that triangulates $D_J$. Let $C^q(K_J)$ be the $q$-th simplicial cochain group of $K_J$. Consider the double complex
\begin{equation}\label{simplicialcech}
C^{p,q}(K) = \bigoplus_{|J| = p}  C^q(K_J).
\end{equation}
The \emph{simplicial \v Cech complex} is the filtered complex defined by 
$\check C^l(K) = \bigoplus_{p+q=l} C^{p,q}(K)$, with filtration
$F^p\check C^l(K) = \bigoplus_{j\geq p}\bigoplus_{j+q=l}C^{j,q}(K)$. The map of double complexes $C^{p,q}\to C^{p,q}(K)$ that is the sum of the chain maps $C^*(D_J)\to C^*(K_J)$ adjoint to the inclusion $C_*(K_J)\to C_*(D_J)$ defines a filtered quasi-isomorphism $\check C^*(\Xb,D)\to \check C^*(K)$ from the \v Cech complex \eqref{cechcomplex} to the simplicial \v Cech complex \eqref{simplicialcech}.


\subsection{The cellular dual complex}\label{cellulargysinsection}
Let $K^*$ be the dual cell complex of the simplicial complex $K$ \cite[\S 64]{munkres}. For each $J\subset I$, the cells of the dual complex $K^*_J$ of the triangulation $K_J$ of $D_J$  are the intersections with $D_J$ of the cells of $K^*$. The dimension of the smooth variety $D_J$ is $n - p$, where $p = |J|$. The double complex $C^{p,q}(K)$ is isomorphic to the double complex
\begin{equation}\label{cellulargysin}
C_{p,k}(K^*) = \bigoplus_{|J| = p} C_{k}(K_J^*)
\end{equation}
via the classical cellular Poincar\' e duality isomorphism $C^q(K_J) \to C_{n-p-q}(K_J^*)$ which assigns to a simplex $\sigma\in K_J$ the dual cell $\sigma^*_J \in K^*_J$.  The second differential of this double complex $\partial'':C_{p,k}(K_J^*)\to C_{p,k-1}(K_J^*)$ is given by the cellular boundary map. 
The first differential $\partial': C_{p,k}(K^*)\to C_{p+1,k-1}(K^*)$  is given by the cellular Gysin maps $i(J,J')^*: C_k(K_J^*)\to C_{k-1}(K_{J'}^*)$, where $|J|= p$, $|J'|= p+1$, and $J\subset J'$. By definition the Gysin map on cellular chains is Poincar\' e dual to the restriction map on simplicial cochains. If $\sigma \in K_{J'}\subset K_J$  then $i(J,J')^*(\sigma^*_J) = \sigma^*_{J'} = \sigma^*_J\cap D_{J'}$. Thus the Gysin map $i(J,J')^*$ applied to a cellular chain in $K^*_J$ is the intersection of the chain with $D_{J'}$.

Thus there is a filtered  chain isomorphism from the simplicial \v Cech complex to the \emph{cellular dual complex} 
$$
\w C_k(K^*) = \bigoplus_pC_{p,k}(K^*)
$$ 
with boundary map $\partial = \partial' +\partial''$ and filtration $F^p\w C_k(K^*) = \bigoplus_{j\geq p}C_{j,k}(K^*)$. The term $(E^1,d^1)$ of the spectral sequence of this filtered complex is the Gysin complex \eqref{Gysincomplex}.


\subsection{The cellular pullback}\label{cellularpullbacksection}
We define a filtered quasi-isomorphism from the cellular dual complex to the corner complex,
$$
\phi: \w C_*(K^*)\to C_*(X')\ .
$$
For $|J|=p$ we define $\phi: C_k(K_J^*)\to C_k(X')$ as follows. For each $k$-cell $B\in C_k(K_J^*)$ let $\phi(B) = [\pi\inv B] \in C_k(X')$. Since $\pi\inv B = \rho_J\inv(\pi_J\inv B)$ we have $[\pi\inv B]\in F^pC_*(X')$.

We claim that $\phi(\partial B)= \partial\phi(B)$ for all $k$-cells $B\in C_k(K_J^*)$, and so $\phi$ is a chain map. By definition $\partial B = \partial'B + \partial''B$. 

If $J'\supset J$ with $|J'|= p+1$ then $\dim(B\cap D_{J'}) = k-1$; in fact, $B\cap D_{J'}$ is a $(k-1)$-cell of $K_{J'}^*$. Thus $\partial'B= \sum_{J'}B\cap D_{J'}$, summed over all $J'\supset J$ with $|J'|=p+1$. Therefore 
$$
\phi(\partial'B) = \sum_{J'}\phi(B\cap D_{J'}) = \sum_{J'}[\pi\inv(B\cap D_{J'})]= [(\pi\inv B)\cap \partial(\pi\inv D_J)],
$$
where $\partial(\pi\inv D_J)=\bigcup_{J'}\pi\inv D_{J'}$ is the boundary of the manifold $\pi\inv D_J$.

Let $\operatorname{bd}(B)$ denote the cellular boundary of the $k$-cell $B$. In other words, if $h:\mathbb B^k\to B$ is a semialgebraic homeomorphism from the unit ball in $\R^k$ onto $B$, then $\operatorname{bd}(B)= h(\partial \mathbb B^k)$.
We have $\partial'' B = \sum B'$, summed over all $(k-1)$-cells $B'$ of $K_J^*$ with $B'\subset \operatorname{bd}(B)$. So 
$$
\phi(\partial''B) = \sum\phi(B') = [\pi\inv\operatorname{bd}(B)] .
$$
On the other hand,
$$
\partial\phi(B) =\partial[\pi\inv B] = [(\pi\inv B)\cap \partial(\pi\inv D_J)] +  [\pi\inv\operatorname{bd}(B)],
$$
which gives the claim.

Finally, we show that the chain map $\phi$ is a filtered quasi-isomorphism. For a $k$-cell $B$ of $K_J^*$, the set $\pi_J\inv B$ is the image of a semialgebraic section of $\pi_J:D_J'\to D_J$ over $B$. Thus the induced map
$$
\phi_p:\frac{F^p\w C_*(K^*)}{F^{p+1}\w C_*(K^*)}= \bigoplus_{|J|=p}C_*(K_J^*)\to \frac{F^p C_*(X')}{F^{p+1} C_*(X')}
$$
factors as $\phi_p= \psi_p\circ \xi_p$, where
$$
\xi_p: \bigoplus_{|J|=p}C_*(K_J^*)\to \bigoplus_{|J|=p}C_*(D_J)
$$
is the inclusion $\xi_p(B) = [B]$, and
$$
\psi_p:\bigoplus_{|J|=p}C_*(D_J)\to \frac{F^p C_*(X')}{F^{p+1} C_*(X')}
$$
is given by Proposition \ref{isotoGysin}. The chain map $\xi_p$ induces an isomorphism in homology, and $\psi_p$ is a chain isomorphism by Proposition \ref{isotoGysin}. Thus $\phi_p$ induces an isomorphism in homology.

This completes the proof of Theorem \ref{cech-corner}.

\subsection{Duality with Borel-Moore homology}\label{dualitysection}

If $X$ is a smooth $n$-dimensional real algebraic variety with good compactification $(\Xb,D)$ and associated corner compactification $X'$, the weight filtration \eqref{weightcomplex} on the complex of semialgebraic chains $C_*(X')$ gives the weight filtration of the classical (compactly supported) homology groups $H_k(X)$, $0\leq k \leq n$,
\begin{equation}\label{homologyweightfiltration}
0 = \mathcal W_{-n-1}H_k(X)\subset \cdots\subset \mathcal W_{-k-1}H_k(X)\subset \mathcal W_{-k}H_k(X) = H_k(X).
\end{equation}
We will show in section \ref{weightsingularsection} that this filtration does not depend on the choice of good compactification of $X$.
In previous work \cite{weight} we defined a weight filtration on the complex of semialgebraic chains with closed supports $C_*^{BM}(X)$  \cite{weight}, which gives the weight filtration of the Borel-Moore homology groups $H^{BM}_{n-k}(X)$, $0\leq k \leq n$,
 \begin{equation}\label{bmhomologyweightfiltration}
 0 = \mathcal W_{-n+k-1}H^{BM}_{n-k}(X)\subset\cdots\subset \mathcal W_{-1}H^{BM}_{n-k}(X) \subset \mathcal W_0H^{BM}_{n-k}(X) = H^{BM}_{n-k}(X).
\end{equation}
  For each $k$, $0\leq k\leq n$, Poincar\' e-Lefschetz duality gives a nonsingular bilinear intersection pairing
$$
\langle\ , \rangle: H_k(X)\times H^{BM}_{n-k}(X) \to \Z_2.
$$
We show that the weight filtrations \eqref{homologyweightfiltration} and \eqref{bmhomologyweightfiltration} on these groups are dual under this pairing.

\begin{thm}\label{Hduality}
Let $X$ be a smooth $n$-dimensional variety. For all $p\leq 0$ and $k\geq 0$,
$$
\mathcal W_pH_k(X) = \{ \alpha \in H_k(X)\ |\ \langle\alpha,\beta\rangle = 0\ \text{for all}\ \beta \in \mathcal W_{-p-n-1}H^{BM}_{n-k}(X) \}.
$$
\end{thm}
\begin{proof}
This is a consequence of a more basic duality of filtered chain complexes. The weight filtration on $H_*(X)$ is induced by the weight filtration on the complex $C_*(X')$, where $X'$ is the corner compactification of $X$. The weight filtration on $C_*(X')$ is by definition the Deligne shift of the corner filtration on $C_*(X')$ \eqref{weightfiltration}. The complex $C_*(X')$ with the corner filtration is in turn filtered quasi-isomorphic to the cohomology \v Cech complex $\check C^*(\Xb,D)$ with its standard filtration (Theorem \ref{cech-corner}). The cohomology \v Cech complex is dual to the homology \v Cech complex $\check C_*(\Xb,D)$, where
$$
\begin{aligned}
&\check C_l(\Xb,D) = \bigoplus_{p+q=l} C_{p,q},\\
&C_{p,q} = \bigoplus_{|J| = p} C_q(D_J),\\
&F_p \check C_l(\Xb,D) = \bigoplus_{j\leq p} \bigoplus_{j+q=l} C_{j,q}.
\end{aligned}
$$
Finally, the complex $C^{BM}_*(X)$ with its weight filtration is quasi-isomorphic to the  homology \v Cech complex $\check C_*(\Xb,D)$ with the Deligne shift of the standard filtration \cite[Theorem (1.1)(2), proof of Proposition (1.9)]{weight}. This last quasi-isomorphism corresponds to the isomorphism $H_l(\Xb,D)\cong H^{BM}_l(X)$.
\end{proof}


\section {Functoriality}\label{functorialitysection}

In this section we prove that the weight filtration is functorial for maps of pairs $(\Xb,X)$, where $\Xb$ is a good compactification of $X$. First we show that a regular map $(\fb,f):(\Xb,X)\to (\Yb,Y)$ induces a semialgebraic map $f':X'\to Y'$ of corner compactifications. The group actions on the principal bundles containing $X'$ and $Y'$ are used to prove that the chain map induced by $f'$ preserves the corner filtration. Finally, we compute the corresponding homomorphism of Gysin complexes.

Let $f:X\to Y$ be a regular map of smooth varieties that extends to a regular map $\fb:\Xb\to\Yb$ of good compactifications. The divisor $D= \Xb\setminus X$ is a finite union of  smooth codimension one subvarieties $D_i$, $i\in I_X$ \eqref{ncd},  and the divisor $E = \Yb\setminus Y$ is a finite union of  smooth codimension one subvarieties $E_j$, $j\in I_Y$. In this section we assume that all the divisors $D_i$ and $E_j$ are \emph{irreducible}.
  
For $x\in \Xb$ let $J(x) = \{i\in I_X\ |\ x\in D_i\}$, and for $y\in \Yb$ let $J(y)=\{j\in I_Y\ |\ y\in E_j\}$.  For every $x\in \Xb$ and $y \in \Yb$, there exist good local coordinates $(U, (u_1,\dots,u_n))$ on $(\Xb,D)$ with $(u_1(x),\dots,u_n(x)) = 0$, and   $(V, (v_1,\dots,v_m))$ on $(\Yb,E)$ with $(v_1(y),\dots,v_m(y)) = 0$.
  
  Let $G_X$ be the group of functions $g: I_X\to \{1, -1\}$, and let $G_Y$ be the group of functions 
 $h: I_Y\to \{1, -1\}$.   Let $G(x) = \{ g\in G_X\ |\ g(i) = 1,\ i\notin J(x)\}$ and 
 $G(y) = \{ h\in G_Y\ |\ h(j) = 1,\ j\notin J(y)\}$.
  
\begin{thm}\label{cornermap}
 Let $X$ and $Y$ be smooth real algebraic varieties with good compactifications $\Xb$ and $\Yb$, and let $f:X\to Y$ be a regular map that extends to a regular map $\fb: \Xb\to \Yb$. Let $X'$ and $Y'$ be the corner compactifications associated to $\Xb$ and $\Yb$. 
There exists a unique continuous semialgebraic map $f':X'\to Y'$ such that $f'|X = f$. 
Moreover 
$$\fb\circ \pi_X = \pi_Y\circ f',$$
 where $\pi_X: X'\to \Xb$ and $\pi_Y:Y'\to \Yb$ are the projections. 
 
 If $Z$ is a smooth real algebraic variety with good compactification $\Zb$, and $\xi:Y\to Z$ is a regular map that extends to a regular map $\overline\xi:\Yb \to \Zb$, then
 $$
 (\xi\circ f)' = \xi'\circ f' .
 $$
\end{thm}
\begin{proof} Given $x'\in X'$, let $ x = \pi_X(x')$ and $ y = \fb( x)$. Choose good coordinate neighborhoods $U$ of $ x$ and $V$ of $y$ as above, with $\fb(U)\subset V$. 
For $g\in G(x)$ and $h\in G(y)$ consider the sets $X_g(U)\subset X$, $\Xb_g(U)\subset \Xb$ and $Y_h(V)\subset Y$, $\Yb_h(V)\subset \Yb$ \eqref{XgU} and \eqref{XbargU}. Let $X'(U)= \pi_X^{-1}U$, $X'_g(U) = \pi_X^{-1}\Xb_g(U)$, and $Y'(V) = \pi_Y^{-1}V$, $Y'_h(V) = \pi_Y^{-1}\Yb_h(V)$. Then $X'(U) = \bigsqcup_g X'_g(U)$ and $Y'(V) = \bigsqcup_hY'_h(V)$. Now $x'\in X'_{g_0}(U)$ for a unique  $g_0\in G(x)$. The open sets $X_g(U)$ and $Y_h(V)$ are connected, so there is a unique $h_0$ with $f(X_{g_0}(U))\subset Y_{h_0}(V)$. Let $y'$ be the unique element of $Y'_{h_0}(V)$ such that $\pi_Y(y') =  y$, and set $f'(x') = y'$. 

By construction $\fb(\pi_X(x') )= \pi_Y(f'(x'))$, so $\fb|X = f$, and the graph of $f'$ is the closure in $X'\times Y'$ of the graph of $f$. Therefore $f'$ is continuous and semialgebraic. The function $f'$ is uniquely determined by $f$ since $X$ is dense in $X'$. It follows that if $f:X\to Y$ and $\xi: Y\to Z$ are as above, then $(\xi\circ f)'= \xi'\circ f'$.
\end{proof}

If $f:X \to Y$ is a regular map of smooth varieties that extends to a regular map $\fb:\Xb\to \Yb$ of good compactifications, with $D = \Xb\setminus X$ and $E=\Yb\setminus Y$, then $\fb^{-1}(E) \subset D$. Let $\fb(x) = y$ and let $U$ and $V$ be good coordinate neighborhoods of $x$ and $y$, respectively, as above, with $\fb(U)\subset V$. Suppose that for $i\in J(x)$ the divisor $D_i\cap U$ of $U$ is given by $u_{k_U(i)}=0$, and for $j\in J(y)$ the divisor $E_j\cap V$ of $V$ is given by $v_{k_V(j)}= 0$. For every $i\in J(x)$ and $j\in J(y)$ there are non-negative integers $a_{ij}$ and a real analytic  function $r_j:U\to \R$ such that on $U$ we have
\begin{equation}\label{quasimonomialdef}
v_{k_V(j)}\circ \overline f = r_j \prod_{i\in I(x)} (u_{k_U(i)})^{a_{ij}}.
\end{equation}
Moreover, $r_j\inv (0) \subset D$ and $\dim r_j\inv (0) \le n-2$, and therefore $r_j$ has constant sign on $U\setminus D$. Since the divisors $D_i$ and $E_j$ are irreducible, the exponents $a_{ij}$ do not depend on the choice of $x$ and $y$ or on the choice of good local coordinates.  
Indeed, they are defined by the condition that the divisor
$$
\fb \inv (E_j) -  \sum_{i\in I_X} a_{ij} D_i  , 
$$
described locally by $r_j=0$, has   real part of dimension strictly less than $n-1$. (See Remark \ref{monomial} for an example.)  Thus the numbers $a_{ij}$ are well-defined not only for the divisors $D_i$ and $E_j$ such that $D_i\cap \fb\inv (E_j) \ne  \emptyset$ by \eqref{quasimonomialdef}, but also we have that 
if $D_i\cap \fb\inv (E_j) = \emptyset$, then  $a_{ij}=0$.

  We define a homomorphism $\varphi: G(x)\to G(y) $ by
\begin{equation}\label{varphi}
\varphi(g)(j) = \prod_{i\in J(x)}  g(i)^{a_{ij}} .
\end{equation}

\begin{prop}\label{prefunctoriality}
Let $f:X\to Y$ be a regular map of smooth varieties that extends to a regular map $\fb:\Xb\to\Yb$ of good compactifications, and let $f':X'\to Y'$ be  the associated map of corner compactifications. If $\fb(x) = y$, 
then 
$$
f'(g\cdot x') = \varphi (g)\cdot f'(x')
$$
for all $g\in G(x)$ and $x'\in \pi_X\inv(x)$. 
\end{prop}

\begin{proof}
Let $f(X_1(U)) \subset Y_h(V)$.  Then by  \eqref{quasimonomialdef} we have  $f(X_g(U)) \subset Y_{\varphi(g)h}(V)$, and this gives the proposition.
\end{proof}

\begin{thm}\label{functoriality} Let $f:X\to Y$ be a regular map of smooth varieties that extends to a regular map $\fb:\Xb\to\Yb$ of good compactifications. If $f':X'\to Y'$ is  the associated map of corner compactifications, then for all $k, p\geq 0$, 
$$
f'_*(F^pC_k(X'))\subset F^pC_k(Y'). 
$$
\end{thm}

\begin{proof}
By Lemma \ref{localtrivialization} and Corollary \ref{nchoosep}, it suffices to show the claim for  $c \in F^pC_k(X')$ of the form 
$$
c=[B] \otimes [G(I)]\in C_k(B)\otimes C_0(G(I')),
$$ 
where $I\subset I'$, $|I|=p$,  and $B$ is the closure of $\mathring B$, with $\mathring B\subset \mathring D_{I'}$.  Suppose, moreover, that  $\fb (\mathring B) \subset \mathring E_{J'}$.  We define a homomorphism $\varphi_{I'J'}: G(I')\to G(J') $ by
$$
\varphi_{I'J'} (g)(j) = \prod_{i\in I'}  g(i)^{a_{ij}} .
$$ 
Suppose first that $\varphi_{I'J'} $ restricted to $G(I)$ is injective.  Then by Proposition \ref{prefunctoriality} we have 
\begin{align}\label{functorial}
f'_* ([B] \otimes [G(I)]) = \fb_* ([B]) \otimes [\varphi_{I'J'} (G(I))],
\end{align}
which lies in  $F^p C_k(Y')$ by Lemma \ref{rank} and Proposition \ref{generating}.  

If $\varphi_{I'J'} $ restricted to $G(I)$ is not injective, then for every $x\in \mathring B$  the fibers of 
$f' :\{ x\} \times G(I) \to \{\fb (x)\} \times G(J')$ have even cardinality.  Therefore the pushforward  
$ f'_* ([B] \otimes [G(I)])$ is equal to $0$.  
\end{proof}

By Proposition \ref{isotoGysin}, $f'$ induces a morphism of complexes 
\begin{equation}\label{secondmorphism}
f_p : \bigoplus_{|I| = p} C_*(D_I) \to\bigoplus_{|J| = p} C_*(E_J).  
\end{equation}
We now show that $f_p$ is a combination of pushforwards with weights.  

To a pair $(I,J)$ with $I\subset I_X$, $J\subset I_Y$, and  $|J|=|I|=p$,  we associate the number $a_{I J} = \det (a_{ij}) _{i\in I, j\in J}$.

\begin{lem}\label{aIJzero}
 Let $D_I^0$ be an irreducible component  of $D_I$, and suppose $D_I^0\cap \fb \inv (E_J) \ne \emptyset$, 
 $|I|=|J|=p$.  
 If $D_I^0 \not \subset \fb \inv (E_J)$  then $a_{IJ}=0$.     
\end{lem}

\begin{proof}
Let $x\in D_I^0$ and $y=  \fb(x)\in E_J$.  Choose good local coordinates 
 $(U, (u_1,\dots,u_n))$ on $(\Xb,D)$ and   $(V, (v_1,\dots,v_m))$ on $(\Yb,E)$ as above, with  $\fb (U) \subset V$, and such that 
 $D_I\cap U = \{u_1 = \cdots =u_p=0\}$ and 
 $E_J\cap V = \{v_1 = \cdots =v_p=0\}$.  If $D_I^0 \not \subset \fb \inv (E_J)$, then by \eqref{quasimonomialdef} there exists $j\in \{1, \ldots ,p\}$ such that $a_{ij} = 0 $ for all  $i\in \{1, \ldots ,p\}$.  Hence $a_{IJ}=0$.  
\end{proof}

\begin{prop}\label{functorialitydyvisors}
Let $f:X\to Y$ be a regular map of smooth varieties that extends to a regular map $\fb:\Xb\to\Yb$ of good compactifications, with the divisors $D= \Xb\setminus X = \bigcup _{i\in I_X} D_i$,  $E = \Yb\setminus Y= \bigcup_{j\in I_Y} E_j$.  Then for every  $I\subset I_X$, $|I|=p$, and for every irreducible component $D_I^0$ of $D_I$, the 
morphism $f_p$ of \eqref {secondmorphism}  restricted to $D_I^0$ is given by 
\begin{equation}\label{forfp}
f_p| D_I^0 = \bigoplus_J  a_{IJ} (\fb_{IJ}^0)_*\ , 
\end{equation}
where  the sum is taken over all  $J\subset I_Y$, $|J|=p$, such that $\fb(D_I^0) \subset E_J$, with $\fb_{IJ}^0:D_I^0\to E_J$ the restriction of $\fb$, and the other components of 
$f_p| D_I^0$ are zero. 
\end{prop}

\begin{proof}
Let  $c\in C_k (D_I^0)$ and suppose that $c=[B] \otimes [G(I)]$, where $B$ is the closure of $\mathring B$ and $\mathring B\subset \mathring D_{I'}$, $I\subset I'$, and  that  $\fb (\mathring B) \subset \mathring E_{J'}$.  
If $\varphi_{I'J'} $ restricted to $G(I)$ is not injective then  $f'_* (c)=0$ and $a_{IJ} =0$ 
for all $J\subset J'$.  

Suppose now that $\varphi_{I'J'} $ restricted to $G(I)$ is injective.    By formula \eqref{functorial} it suffices to decompose the image of 
$\varphi_{I'J'} (G(I))$ in $\mathcal I^p (G(J'))/ \mathcal I^{p+1} (G(J'))$ with respect to the basis given by Corollary \ref{nchoosep}.   
Then \eqref{forfp} follows from Corollary \ref{morphism} both in the case when $\fb(D_I^0)  \subset E_J$ 
and when $\fb(D_I^0) \not  \subset E_J$.  Indeed, in the latter case the claim follows from  the fact that $a_{IJ} =0$  by Lemma \ref{aIJzero}.  
\end{proof}

Recall that a good compactification gives rise to a Gysin complex defined by  \eqref{Gysincomplex}.  Thus $\fb:\Xb\to\Yb$ induces a morphism of Gysin complexes $G(\Xb, X) \to G(\Yb, Y)$ that can be computed  
using  Proposition  \ref{functorialitydyvisors}.   We will encounter in the following sections several examples of morphisms of Gysin complexes that are simply the \emph{homology pushforward} 
\begin{equation*}
\bigoplus_{|I| = p} H_*(D_I) \to\bigoplus_{|J| = p} H_*(E_J)
\end{equation*}
given by the sum  of all the induced maps $D_I \to E_J$.   

\begin{cor}\label{pushforward}
Let $f:X\to Y$ be a regular map of smooth varieties that extends to a regular map $\fb:\Xb\to\Yb$ of good compactifications.   Suppose that  for all $p$, and all $I\subset I_X$, $J\subset I_Y$ such that $|I|=|J|=p$,  either
\begin{enumerate}
\item 
$\fb(D_{I}) \subset E_J$ and then $a_{IJ}=1$, or
\item
$\dim D_I\cap \fb \inv ( E_J) < \dim D_I$.
\end{enumerate}
Then  the induced morphism of Gysin complexes $G(\Xb, X) \to G(\Yb, Y)$  
is the  homology pushforward.  
\end{cor}
\begin{proof}
This is an immediate consequence of  Corollary \ref{E1Gysin} and Proposition  \ref{functorialitydyvisors}.  
\end{proof} 

\begin{rem}\label{monomial}
 We say that $\fb$ is a \emph{monomial map} (with respect to the divisors $D$ and $E$) if for every $x\in \Xb$ and $y = \fb(x)$, the functions $r_{j}$ of \eqref {quasimonomialdef} are never zero.  
Then $\fb:\Xb\to \Yb$ is a topological \emph{tico map} with respect to the ticos $D$ and $E$ \cite[III.2]{akbulutking}. (``Tico'' stands for ``transversely intersecting codimension one.'') In this case the coefficients $a_{ij}$ can be simply defined by
$$
\fb \inv (E_j)= \sum _i a_{ij} D_i\ .
$$
In the complex algebraic case a regular map $\fb : (\Xb,\Xb\setminus D) \to (\Yb,\Yb\setminus E)$, where $D$ and $E$ are divisors with normal crossings, is automatically monomial in this sense \cite[p~176]{akbulutking}.  This is not true in the real algebraic case, as the following example shows. 

Let $X= \R^2$, $Y = \R$, and $f:\R^2\to \R$, $f(x,y) = x^2/(1 + y^2)$. Let $\Xb = \mathbb P^2$ with coordinates $[x:y:z]$, so that $D = \{z=0\}$, and let $\Yb = \mathbb P^1$ with coordinates $[s:t]$, so that $E = \{t=0\}$. Let $\fb:\mathbb P^2 \to \mathbb P^1$, $\fb[x:y:z] = [x^2:y^2+z^2]$. Then $\fb \inv (E)$ is a single point of $D$.  
\end{rem}

\section{Blowup squares}\label{BlowupSquaresection}

In this section we analyze the homology of a classical blowup square. A key tool is the Leray-Hirsch theorem on the homology of a projectivized vector bundle. In sections \ref{Blowuptransversesection} and \ref{Blowupcontainedsection} we will apply this special case to understand the behaviour of the weight filtration under a blowup with smooth center contained in a good compactification of a smooth variety.

A \emph{blowup square} (also called an \emph{elementary acyclic square}) is a cartesian diagram 
of  compact irreducible nonsingular real algebraic varieties and regular morphisms 
\begin{equation}\label{square}
\renewcommand{\arraystretch}{1.5}
\begin{array}[c]{ccc}
E&\stackrel{\displaystyle s}\longrightarrow&\w M\\ 
\hphantom{q}\downarrow{q}&&\hphantom{p}\downarrow{p} \\ 
C&\stackrel{\displaystyle r}\longrightarrow&M
\end{array}
\end{equation}
such that $C$ is a subvariety of $M$ with inclusion $r$, $\w M$ is the blowup of $M$ with center $C$ 
and projection $p$, and $E = p\inv(C)$ is the exceptional divisor.

In what follows we suppose $\dim C< \dim M$.

\begin{lem}\label{degree1}
The composition $p_*\circ p^*$ is the identity map, so
 $p_*:H_*(\w M)\to H_*(M)$ is surjective and $p^*:H_*(M)\to H_*(\w M)$ is injective.
 
\end{lem}
\begin{proof}
If $\alpha\in H_*(M)$, by Poincar\' e duality there exists $\beta\in H^*(M)$ with $\alpha=\beta\frown[M]$. Then $p^*(\alpha) = p^*(\beta) \frown [\w M]$, and  $p_*[\w M] = [M]$ since $p$ has degree 1. Thus 
$$
p_*p^*(\alpha) = p_*(p^*(\beta) \frown [\w M]) = \beta \frown p_*[\w M] = \beta \frown [M] = \alpha. 
$$

\end{proof}

\begin{prop}\label{shortexactsequence}
Given a blowup square \eqref{square}, for every $k>0$ there is a short exact sequence
\begin{equation}\label{shorthomology}
0\rightarrow H_k(E)\stackrel{\displaystyle i_*}\rightarrow  H_k(C)\oplus H_k(\w M) \stackrel{\displaystyle j_*}\rightarrow   H_k(M)\to 0,
\end{equation}
where $i_*(\alpha) = (q_*(\alpha),s_*(\alpha))$ and $j_*(\beta,\gamma) = r_*(\beta) + p_*(\gamma)$.  Moreover,  
$q_*$ is surjective and $s_*$ induces an isomorphism $\ker q_*  \stackrel{ \simeq }\longrightarrow \ker p_* $
\end{prop}

\begin{proof}
Consider the commutative diagram
\begin{equation*}
\renewcommand{\arraystretch}{1.5}
\begin{array}[c]{ccccccccc}
 \longrightarrow & H_{k+1} (\w M, E) &\longrightarrow  & H_k(E) &\stackrel{\displaystyle s_k}
\longrightarrow& H_k(\w M) & \longrightarrow & H_k(\w M, E) &\longrightarrow  \\
&  \hphantom{q_*}\downarrow{p'_{k+1}} && \hphantom{q_*}\downarrow{q_k} && \hphantom{p_*}\downarrow{p_k} &&\hphantom{p'_*}\downarrow{p'_k}  &\\
 \longrightarrow  & H_{k+1} (M, C) &\longrightarrow  & H_k(C) &\stackrel{\displaystyle r_k}\longrightarrow& H_k(M) & \longrightarrow & H_k(M, C) &\longrightarrow 
\end{array}
\end{equation*}
The rows are exact, the maps $p'_k$ are isomorphisms, and the maps $p_k$ are surjective by  Lemma \ref{degree1}. The proposition is proved by a diagram chase.
\end{proof}

The exactness of the sequence \eqref{shorthomology}  can be paraphrased by saying that the square 
\begin{equation}\label{Homologyquare}
\renewcommand{\arraystretch}{1.5}
\begin{array}[c]{ccc}
H_{k} (E) &\stackrel{\displaystyle s_*}\longrightarrow& H_k(\w M)\\ 
\hphantom{ q}\downarrow{q_*}&&\hphantom{p}\downarrow{p_*} \\ 
H_{k} (C) &\stackrel{\displaystyle r_*}\longrightarrow& H_k( M) 
\end{array}
\end{equation}
is commutative and acyclic.  

\begin{cor}\label{coshortexactsequence}
Given a blowup square \eqref{square}, for every $k>0$ there is a short exact sequence
\begin{equation}\label{shortcohomology}
0\leftarrow H^k(E)\stackrel{ \displaystyle i^*}\leftarrow  H^k(C)\oplus H^k(\w M) \stackrel{\displaystyle j^*}\leftarrow   H^k(M)\leftarrow 0,
\end{equation}
where $i^*(\beta,\gamma) = q^*(\beta)+s^*(\gamma)$ and $j^*(\delta) = (r^*(\delta),  p^*(\delta))$. 
Moreover,  
$q^*$ is injective and $s^*$ induces an isomorphism $\im q^*  \stackrel{ \simeq }\longleftarrow \im p^* $.
\end{cor}

The exactness of the sequence \eqref{shortcohomology}  says that the square 
\begin{equation}\label{Cohomologysquare}
\renewcommand{\arraystretch}{1.5}
\begin{array}[c]{ccc}
H^{k} (E) &\stackrel{\displaystyle s^*}\longleftarrow& H^k(\w M)\\ 
\hphantom{ q}\uparrow{q^*}&&\hphantom{p}\uparrow{p^*} \\ 
H^{k} (C) &\stackrel{\displaystyle r^*}\longleftarrow& H^k( M) 
\end{array}
\end{equation}
is commutative and acyclic. Equivalently, if  $\dim M-\dim C=m>0$,  the square of Gysin homomorphisms 
\begin{equation}\label{pregysinsquare}
\renewcommand{\arraystretch}{1.5}
\begin{array}[c]{ccc}
H_{k-1} (E) &\stackrel{\displaystyle s^*}\longleftarrow& H_k(\w M)\\ 
\hphantom{ q}\uparrow{q^*}&&\hphantom{p}\uparrow{p^*} \\ 
H_{k-m} (C) &\stackrel{\displaystyle r^*}\longleftarrow& H_k( M) 
\end{array}
\end{equation}
is commutative and acyclic.

\begin{lem}\label{directsum}{\hfil}
\begin{enumerate}
\item $H_k(\w M) = \ker p_* \oplus \im p^*$. 
\item $H_{k-1}(E) = \im q^* \oplus s^*(\ker p_*)$ .  
\end{enumerate}
\end{lem}

\begin{proof}
(1) follows from Lemma \ref{degree1}.  We prove (2) as follows.
Let $\alpha \in H_{k-1}(E)$.  By Corollary \ref{coshortexactsequence} there are $\beta \in H_{k-m}(C)$  and $\gamma \in H_k(\w M)$ and such that 
$\alpha = q^*(\beta)+s^*(\gamma)$.  Then $\gamma_1 = \gamma -p^*p_* (\gamma) \in \ker p_*$ and 
$\alpha = q^*(\beta +r^*p_*(\gamma))+s^*(\gamma_1)$.  
 If $q^*(\beta)+s^*(\gamma)= 0$ then, by Corollary \ref{coshortexactsequence}, $\beta \in \im r^*$ and  $\gamma \in \im p^*$.  
If, moreover, $\gamma \in \ker p_*$, then since 
 $\ker p_* \cap \im p^* =0$ we have $\gamma =0$.  
\end{proof}

\begin{thm}\label{makingGysinsquare}
Let $m = \dim M-\dim C$. For all $k > 0$  there is a unique homomorphism $\w q_* :H_{k-1} (E) \to H_{k-m} (C)$ such that $\w q_*\circ q^*$ is the identity and the following diagram is commutative and acyclic:
\begin{equation}\label{Gysinsquare}
\renewcommand{\arraystretch}{1.5}
\begin{array}[c]{ccc}
H_{k-1} (E) &\stackrel{\displaystyle s^*}\longleftarrow& H_k(\w M)\\ 
\hphantom{q}\downarrow{\w q_*}&&\hphantom{p}\downarrow{p_*} \\ 
H_{k-m} (C) &\stackrel{\displaystyle r^*}\longleftarrow& H_k( M) 
\end{array}
\end{equation}    
\end{thm}

\begin{proof}

By Lemma \ref{directsum}, $\alpha \in H_{k-1}(E)$ can be written uniquely $\alpha = q^*(\beta)+s^*(\gamma)$, where $\beta \in 
H_{k-m}(C)$ and $\gamma \in \ker p_*$. We require $\w q_*q^*(\beta) = \beta$, and $\w q_*s^*(\gamma) = r^*p_*(\gamma) =0$, so we must have $\w q_*(\alpha) = \beta$, and $\beta$ is unique since $q^*$ is injective. Straightforward computations using Lemma \ref{directsum} and Corollary \ref{coshortexactsequence}  give that 
\eqref{Gysinsquare} is commutative and the associated simple complex is exact.   
\end{proof}

In the blowup square \eqref{square}, the map $q: E\to C$ is the projectivization of the normal bundle of $C$ in $M$.
To give a geometric description of the homomorphism $\w q_*$ we apply the classical \emph{Leray-Hirsch Theorem}: 

\begin{thm}\label{lerayhirsch}
Let $A\to B$ be vector bundle of rank $m$, and let $\pi: \mathbb P(A)\to B$ be its projectivization. Let $e\in H^1(\mathbb P(A))$ be the Euler class of the tautological line bundle. The cohomology group $H^*(\mathbb P(A))$ is a free module over $H^*(B)$ with basis $1, e, e^2,\dots, e^{m-1}$. In other words, every element $u\in H^*(\mathbb P(A))$ can be written uniquely
$$
u = \pi^*(u_0) + \pi^*(u_1)\smile e + \cdots + \pi^*(u_{m-1})\smile e^{m-1},
$$
where $u_0,u_1,\dots,u_{k-1}\in H^*(B)$.
\end{thm}
\begin{proof}
The proof uses the Leray-Serre spectral sequence 
\cite[Theorem 5.10, p.~48]{mccleary}.
\end{proof}

If the base $B$ of the vector bundle is a topological manifold of dimension $b$, then $\mathbb P(A)$ is a manifold of dimension $b+m-1$, and by Poincar\' e duality the Leray-Hirsch Theorem gives that every element $\alpha\in H_*(\mathbb P(A))$ can be written uniquely
\begin{equation}\label{lh}
\alpha = \pi^*(a_0) + e\frown \pi^*(a_1) + \cdots + e^{m-1}\frown \pi^*(a_{m-1}),
\end{equation}
where $a_0, a_1,\dots,a_{m-1}\in H_*(B)$.

\begin{lem}\label{projection}
If $\pi:\mathbb P(A)\to B$ is the projectivization of an $m$-plane bundle and $\alpha\in H_*(\mathbb P(A))$ is given by \eqref{lh}, then $\pi_*(\alpha) = a_{m-1}$.
\end{lem}
\begin{proof}
Suppose that $\alpha = e^i\frown \pi^*(a_i)$, $i\in \{0,\dots,m-1\}$, and $a_i = u_i\frown [B]$. Let $\varepsilon^i = e^i\frown [\mathbb P(A)] $.
Then $$
\alpha= e^i\frown(\pi^*(u_i)\frown [\mathbb P(A)])= \pi^*(u_i)\frown (e^i\frown [\mathbb P(A)]) = \pi^*(u_i)\frown \varepsilon^i
$$
and so 
$$
\pi_* (\alpha) = \pi_*\left (\pi^*(u_i)\frown \varepsilon^i\right )= u_i\frown \pi_*(\varepsilon^i).
$$
Now $\varepsilon^i\in H_{b+m-1-i}(\mathbb P(A))$, and if $i<m-1$ then $b+m-1-i> b$, so $\pi_*(\varepsilon^i) = 0$. 

On the other hand, we claim that $\pi_*(\varepsilon^{m-1}) = [B]$, and so if $\alpha = e^{m-1}\frown \pi^*(a_{m-1})$ then $\pi_*(\alpha) = u_{m-1}\frown [B]=a_{m-1}$. 
Now $\pi_*(\varepsilon^{m-1})\in H_b(B)$, and we have $\pi_*(\varepsilon^{m-1}) = [B]$ if and only if $\rho_x(\pi_*(\varepsilon^{m-1}))\neq 0$ for all $x\in B$, where $\rho_x: H_b(B)\to H_b(B,B\setminus \{x\})$ is the restriction map. There is a commutative square
\begin{equation*}
\renewcommand{\arraystretch}{1.5}
\begin{array}[c]{ccc}
H_b(\mathbb P(A)) &\stackrel{\displaystyle \sigma_x}\longrightarrow&H_b(\mathbb P(A), \mathbb P(A)\setminus \pi\inv(x))\\ 
\hphantom{\pi_*}\downarrow{\pi_*}&&\hphantom{\pi_*}\downarrow{\pi_*} \\ 
H_b(B)&\stackrel{\displaystyle \rho_x}\longrightarrow&H_b(B,B\setminus \{x\})
\end{array}
\end{equation*}
with $H_b(\mathbb P(A), \mathbb P(A)\setminus \pi\inv(x))= H_b(B,B\setminus \{x\})\otimes H_0(\pi\inv(x))$ and $\pi_*(\alpha\otimes \beta) = \phi(\beta)\alpha$,  where $\phi:H_0(\pi\inv(x))\to \Z_2$ is the augmentation isomorphism. By the local triviality of the bundle $\pi: \mathbb P(A)\to B$, we have 
\begin{equation*}
\begin{aligned}
\sigma_x(\varepsilon^{m-1}) &= \sigma_x(e^{m-1}\frown [\mathbb P(A)])\\
&= \rho_x[B]\otimes \left ((e^{m-1}|\pi\inv(x))\frown [\pi\inv(x)]\right )\\ &= \rho_x[B]\otimes \left ((e|\pi\inv(x))^{m-1}\frown [\pi\inv(x)]\right).
\end{aligned}
\end{equation*}
Now $\pi\inv(x) = \mathbb P^{m-1}$, and $e|\pi\inv(x)$ is the Euler class of the tautological line bundle, so $(e|\pi\inv(x))^{m-1}\neq 0$, and hence $\rho_x(\pi_*(\varepsilon^{m-1})) \neq 0$.
\end{proof}

Now we show that the homomorphism $\w q_* :H_{k-1} (E) \to H_{k-m} (C)$ of \eqref{Gysinsquare} can be defined geometrically in terms of the excess bundle, which is defined as follows.  
Let $\mathcal N_C$ be the normal bundle of $C$ in $M$ and denote by $e(C)\in H^m(C)$ its Euler class (\emph{i.e.}\ the top 
Stiefel-Whitney class).  Similarly we denote by $\mathcal N_E$  the normal bundle of $E$ in $\w M$ and 
 by $e(E)\in H^1(E)$ its Euler class. Then $q:E\to C$ is the projectivization of $\mathcal N_C$, and $\mathcal N_E$ is the tautological line bundle. The \emph{excess bundle} is the quotient bundle $\mathcal E = q^* \mathcal N_C /\mathcal N_E$.  The Euler class $e(\mathcal E)$ satisfies $q^*e(C) = e(\mathcal E) e(E)$. 
 
 \begin{prop}
Let $\alpha \in H_{k-1}(E)$.  Then $\w q_*(\alpha) = q_* (e(\mathcal E) \frown \alpha)$.  
 \end{prop} 
 
\begin{proof}
Since $\alpha = q^*(\beta)+s^*(\gamma)$, where $\beta \in 
H_{k-m}(C)$ and $\gamma \in \ker p_*$,  it suffices to consider two cases, $\alpha = q^*(\beta)$ or $\alpha = s^*(\gamma)$ with $\gamma \in \ker p_*$.  

If $\alpha = q^*(\beta)$ then we have to show that  $q_* (e(\mathcal E) \frown q^*(\beta))= \beta$.  
The Whitney formula for the total Stiefel-Whitney class $w(q^* \mathcal N_C) = w( \mathcal N_E) w(\mathcal E) = (1 + e(E))w(\mathcal E)$
yields 
$$
e(\mathcal E) = w_{m-1} (\mathcal E) = \sum_{i=0} ^{m-1}  e(E) ^{m-1-i}\smile q^*( w_i ( \mathcal N_C)).
$$
Therefore by  Lemma \ref{projection},
\begin{align*}
q_* (e(\mathcal E) \frown q^*(\beta))
&=  q_*\left (\sum_{i=0} ^{m-1} ( e(E) ^{m-1-i} \smile q^*( w_i ( \mathcal N_C)))\frown q^*(\beta))\right ) \\ 
&=  \sum_{i=0} ^{m-1} q_*\left ( e(E) ^{m-1-i} \frown (q^*( w_i ( \mathcal N_C))\frown q^*(\beta))\right ) \\ 
 &= q_*\left (  e(E) ^{m-1}\frown( q^*( w_0 ( \mathcal N_C)\frown q^*(\beta))\right )\\
 & = \beta.  
\end{align*}

If $\alpha = s^*(\gamma)$ with $\gamma \in \ker p_*$,  then we have to show that  
$q_* (e(\mathcal E) \frown s^*(\gamma))= 0$.  By Proposition \ref{shortexactsequence}  there is $\w \alpha \in H_{k} (E)$
 such that $\gamma =  s_*(\w \alpha )$ and 
$q_*(\w \alpha ) =0$.  Therefore $\alpha =  s^*(\gamma) =  s^*(s_*(\w \alpha )) = e(E) \frown \w \alpha$, and so we have
\begin{align*}
q_* (e(\mathcal E) \frown s^*(\gamma)) &= q_* ((e(\mathcal E) \smile e(E) ) \frown \w \alpha)\\  &=  q_* (q^* (e(C) ) \frown \w \alpha)\\ &= e(C) \frown q_* (\w \alpha)\\ &= 0.  
\end{align*}
This completes the proof of the Proposition.
\end{proof} 


\section{Blowup with center transverse to the divisor at infinity}\label{Blowuptransversesection}

In section \ref{weightsingularsection} we will apply the main theorem of Guill\' en and Navarro Aznar \cite{navarro} to extend our weight filtration to singular varieties. Their key  extension criterion describes the behavior of the weight filtration for the blowup of a good compactification with center transverse to the divisor at infinity. In the present section we verify this extension criterion. The key result is the acyclicity of the Gysin diagram \eqref{gysinsquare} of a blowup square.

Let $X$ be a smooth $n$-dimensional variety and let $W =\Xb$ be a good compactification of $X$, with $W\setminus X = D$ a divisor with normal crossings, so that $D =\bigcup_{i\in I} D_i$, where $D_i$ are smooth hypersurfaces meeting transversely. Let $Z$ be an irreducible smooth $m$-dimensional subvariety of $X$ and let $Y= \Zb$ be the closure of $Z$ in $W$. Suppose that $Y$ is a smooth subvariety of $W$ such that $Y$ has normal crossings with $D$ \cite[(2.3.1)]{navarro} and $Y\not\subset D$. Then for every $x\in W$ there is a good local coordinate system $(U,(u_1,\dots,u_n))$ about $x$, and for each $i\in J(U)$ there is an index $k(i)\in \{1,\dots,n\}$, $k(i)\leq m$, such that $D_i\cap U$ is the coordinate hyperplane $u_{k(i)}=0$, and $Y\cap U$ is given by $u_{m+1} = \cdots = u_n = 0$. Thus $Y$ is transverse to the divisor $D$ \cite[III.3]{akbulutking}.

From this data we obtain the blowup square of pairs $(W_\bullet, X_\bullet)$:
\begin{equation}
\label{square1}
\renewcommand{\arraystretch}{1.5}
\begin{array}[c]{ccc}
(\w Y,\w Z)&\longrightarrow&(\w W,\w X)\\ 
\downarrow&&\hphantom{b}\downarrow{ b} \\ 
(Y, Z)&\stackrel {\displaystyle a}\longrightarrow&(W,X)
\end{array}
\end{equation}
Here $a$ is the inclusion, $b$ is the blowup of $(W,X)$ along $(Y,Z)$, and $(\w Y,\w Z)= b\inv(Y,Z)$. Since $Y$ has normal crossings with $D$, it follows that $\w W$,  $Y$, and $\w Y$ are good compactifications of $\w X$, $Z$, and $\w Z$, respectively. 

\begin{thm}\label{blowup1} \emph{Blowup with center transverse to the divisor at infinity.}
Given a blowup square of pairs \eqref{square1}, the corresponding square of corner compactifications induces an acyclic diagram of weight complexes
\begin{equation*}
\renewcommand{\arraystretch}{1.5}
\begin{array}[c]{ccc}
\mathcal WC_*(\w Z')&\longrightarrow&\mathcal WC_*(\w X')\\ 
\downarrow&&\hphantom{b'_*}\downarrow{b'_*} \\ 
\mathcal WC_*(Z')&\stackrel{\displaystyle a_*'}\longrightarrow&\mathcal WC_*(X')
\end{array}
\end{equation*}
In other words, the simple filtered complex of this diagram  is quasi-isomorphic to the zero complex.
\end{thm}

For the definition of the simple filtered complex of a diagram of filtered complexes, see \cite[p.~125]{weight}.

Recall that the weight filtration $\mathcal W_*$ \eqref{weightfiltration} is the Deligne shift of the filtration $\widehat F_*$, where $\widehat F_{-p} = F^p$, and $F^*$ is the corner filtration \eqref{cornercomplex} \eqref{cornerss} \eqref{weightfiltration}. Thus to prove the theorem it suffices to show that the spectral sequence of the simple filtered complex associated to the diagram
\begin{equation}
\label{cornersquare}
\renewcommand{\arraystretch}{1.5} 
\begin{array}[c]{ccc}
(C_*(\w Z'),F^*)&\longrightarrow& (C_*(\w X'), F^*)\\ 
\downarrow&&\hphantom{b'_*}\downarrow{b'_*} \\ 
(C_*(Z'), F^*)&\stackrel{\displaystyle a_*'}\longrightarrow&(C_*(X'),F^*)
\end{array}
\end{equation}
has trivial $E^2$ term. This in turn is equivalent to the statement that the simple complex $\simple G(W_\bullet, X_\bullet)$ associated to the diagram of  Gysin complexes
\begin{equation}
\label{gysinsquare}
\renewcommand{\arraystretch}{1.5}
\begin{array}[c]{ccc}
G(\w Y,\w Z)&\longrightarrow&G(\w W,\w X)\\ 
\downarrow&&\hphantom{b_*}\downarrow{b_*} \\ 
G(Y, Z)&\stackrel{\displaystyle a_*}\longrightarrow&G(W,X)
\end{array}
\end{equation}
is acyclic. By Corollary \ref{pushforward} the arrows in \eqref {gysinsquare} are homology pushforward. We will prove that the complex $\simple G(W_\bullet, X_\bullet)$ is acyclic by induction on the \emph{complexity} of the divisor $D = W\setminus X$, which is defined as follows.

Let $D$ be a divisor of the compact nonsingular variety $W$, and suppose that $D$ has simple normal crossings \eqref{ncd}. A \emph{nonsingular decomposition} of $D$ is a set $\mathcal D = \{D_i\}_{i\in I}$ of nonsingular divisors of $W$ such that $D = \bigcup_{i\in I} D_i$. The \emph{complexity} $c(D)$ of the divisor $D$ is the minimum cardinality of a nonsingular decomposition of $D$.

If the divisor $D$ has simple normal crossings in $W$, there is a one-to-one correspondence between nonsingular decompositions of $D$ and partitions of the set $\mathcal C(D)$ of irreducible components of $D$ such that if $C_i$ and $C_j$ belong to the same member of the partition then $C_i\cap C_j = \emptyset$. The nonsingular decomposition $\mathcal D = \{D_i\}_{i\in I}$ corresponds to the partition $\{\mathcal D_i\}_{i\in I}$ of $\mathcal C(D)$, where $\mathcal D_i = \{C_j\ |\ C_j\subset D_i\}$.

\begin{rem}
If $D$ is a simple normal crossing divisor of $W$, let $\Gamma(D)$ be the corresponding graph. The vertices of $\Gamma(D)$ are the irreducible components of $D$, and there is an edge of $\Gamma(D)$ between $C_i$ and $C_j$ if and only if $C_i\cap C_j\neq \emptyset$.
Thus nonsingular decompositions of $D$ are in one-to-one correspondence with \emph{graph partitions} of $\Gamma(D)$, and the complexity $c(D)$ is the \emph{chromatic number} of $\Gamma(D)$.
\end{rem}

Now the inductive proof of Theorem \ref{blowup1} proceeds as follows. In the base case $c(D)= 0$ the divisor $D$ is empty, and the diagram \eqref{gysinsquare} reduces to
\begin{equation*}
\renewcommand{\arraystretch}{1.5}
\begin{array}[c]{ccc}
H_*(\w Y)&\longrightarrow&H_*(\w W)\\ 
\downarrow&&\downarrow \\ 
H_*(Y)&\longrightarrow&H_*(W)
\end{array}
\end{equation*}
which is acyclic by Proposition \ref{shortexactsequence}. 

Now suppose that $c(D) > 0$. Let $\mathcal D = \{D_i\}_{i\in I}$ be a nonsingular decomposition of $D$ with $|\mathcal D| = c(D)$.  Let $D = D''\cup V$ and $D' = D''\cap V$, where $V= D_0\in \mathcal D$. The cubical diagram $(D_J)_{J\subset I} \to W$ ($J\neq\emptyset$) is equal to the diagram
\begin{equation*}
\renewcommand{\arraystretch}{1.5}
\begin{array}[c]{ccc}
(D'_J)_{0\notin J}&\longrightarrow&V\\ 
\downarrow&&\downarrow \\ 
(D''_J)_{0\notin J}&\longrightarrow&W
\end{array}
\end{equation*}
where the vertical maps are inclusions. It follows from the definition of the homological Gysin complex that this diagram yields a short exact sequence of chain complexes,
$$
0\to G(V,V\setminus D')[1]\to G(W,W\setminus D)\to G(W,W\setminus D'')\to 0.
$$
Blowing up along $Y$ transverse to $D$ we obtain a short exact sequence of chain complexes
$$
0\to \simple G(V_\bullet,(V\setminus D')_\bullet)[1]\to \simple G(W_\bullet,(W\setminus D)_\bullet)\to \simple G(W_\bullet,(W\setminus D'')_\bullet)\to 0.
$$
Now $\mathcal D''= \mathcal D\setminus \{V\}$ is a nonsingular decomposition of $D''$ with $|\mathcal D''|= |\mathcal D| - 1 = c(D) - 1$, so $c(D'') < c(D)$. Also $\mathcal D'=\{D_i\cap V\ |\ i\neq 0\}$ is a nonsingular decomposition of $D'$ with $|\mathcal D'|= |\mathcal D| - 1$, so $c(D') < c(D)$. Thus by induction on $c(D)$ the complexes $\simple G(V_\bullet,(V\setminus D')_\bullet)$  and $\simple G(W_\bullet,(W\setminus D'')_\bullet)$
are acyclic. It follows that 
$\simple G(W_\bullet,(W\setminus D)_\bullet)$ is acyclic, as desired. This completes the proof of Theorem \ref{blowup1}.


\section{Blowup with center contained in the divisor at infinity}\label{Blowupcontainedsection}

To see that the weight filtration of a smooth variety $X$ does not depend on the choice of a good compactification $\Xb$ we show that the weight filtration is invariant, up to quasi-isomorphism, under a blowup of $\Xb$ with center contained in the divisor $D$ at infinity. This follows from the fact that the corresponding homomorphism of Gysin complexes induces an isomorphism in homology.

Again let $W =\Xb$ be a good compactification of the smooth variety $X$, and let $D= W\setminus X$. Let $Y$ be an irreducible smooth $m$-dimensional subvariety of $W$ such that $Y\subset D$, and suppose that $Y$ has normal crossings with $D$. Thus for every $x\in W$ there is a good coordinate system $(U, (u_1,\dots, u_n))$ about $x$, with $Y\cap U$  given by $u_{m+1} = \cdots = u_n = 0$, such that for each $i\in J(U)$ there is an index $k(i)\in \{1,\dots,n\}$ with $D_i\cap U$  the coordinate hyperplane $u_{k(i)}=0$,  and there exists $i\in I$ such that $k(i)> m$. Thus $Y$  intersects the divisor $D$ \emph{cleanly} \cite[III.3]{akbulutking}.

From this data we obtain the square
\begin{equation}
\label{square2}
\renewcommand{\arraystretch}{1.5}
\begin{array}[c]{ccc}
(\w Y,\emptyset)&\longrightarrow&(\w W,\w X)\\ 
\downarrow&&\hphantom{b}\downarrow{b} \\ 
(Y, \emptyset)&\stackrel{\displaystyle a}\longrightarrow&(W,X)
\end{array}
\end{equation}
where $a$ is the inclusion, $b$ is the blowup of $W$ along $Y$ (so $b$ maps $\w X$ isomorphically onto $X$), and $\w Y= b\inv(Y)$. Since $Y$ has normal crossings with $D$, it follows that $\w W$ is a good compactification of $\w X$. 

\begin{thm}\label{blowup2} \emph{Blowup with center contained in the divisor at infinity, clean intersection.} 
Given a blowup square of pairs \eqref{square2}, the homomorphism 
$$
b_*': \mathcal WC_*(\w X') \to \mathcal WC_*(X')
$$
is a quasi-isomorphism of filtered complexes.
\end{thm}

By definition of the weight filtration, to prove the theorem it suffices to show that the homomorphism of corner complexes
$$
b_*':(C_*(\w X'), F^*)\to (C_*(X'),F^*)
$$
induces an isomorphism on the $E^2$ term of the corner spectral sequence \eqref{cornerss}. This is equivalent to the statement that the corresponding homomorphism of Gysin complexes 
$$
b_*:G(\w W, \w X) \to G(W, X)
$$
induces an isomorphism in homology. 

First we prove the special case when the divisor $D$ is nonsingular. This is the most involved part of the proof.
 
Let $W = \Xb$ be a good compactification of the smooth variety $X$, and suppose that $D= W\setminus X$ 
is a non-singular divisor in $W$. Let $Y$ be an irreducible smooth $m$-dimensional subvariety of $W$ such that $Y\subset D$.  We assume that the codimension 
of $Y$ in $W$ is bigger than $1$.  
From this data we obtain a blowup square of pairs \eqref{square2}, where the divisor 
$\w D=\w{W} \setminus \w X$ is the union of the proper transform $\widehat D$ of the divisor $D$  and the divisor $\w Y$; \emph{i.e.} $\w D = \widehat D \cup \w Y$.  We let $ \widehat E=\widehat D \cap \w Y$.

\begin{prop}\label{keyquasiisomorphism}
The following diagram is commutative and acyclic,
\begin{equation}\label{keyacyclic1}
\renewcommand{\arraystretch}{1.5}
\begin{array}[c]{ccccccc}
 H_{k} (\w W) &\stackrel{\displaystyle (s^*, \w a^*)}\longrightarrow  & H_{k-1}(\w Y) \oplus H_{k-1}(\widehat D) &\longrightarrow& H_{k-2}(\widehat E)  \\
\hphantom{q_*}\downarrow{p_*} &&p_{1*}\downarrow{p_{2*}} & &\hphantom{p'_*}\downarrow{}  \\
 H_{k} (W ) &\stackrel{\displaystyle a^*}\longrightarrow  & H_{k-1}(D)  & \longrightarrow & 0 
\end{array}
\end{equation}
where the horizontal arrows are Gysin morphisms and the vertical arrows are pushforward maps induced by $p$.   \end{prop}

\begin{proof}
The bottom row of diagram \eqref{keyacyclic1} is the Gysin complex of $(W,W\setminus D)$ and the top row is the Gysin complex of $(\w W,\w W \setminus \w D)$.  Thus the commutativity of \eqref{keyacyclic1} follows from Corollary \ref{pushforward}.

To show the acyclicity of \eqref{keyacyclic1} we consider the following augmented version of \eqref{keyacyclic1}, 
\begin{equation}\label{keyacyclic2}
\renewcommand{\arraystretch}{1.5}
\begin{array}[c]{ccccccc}
 H_{k} (\w W) &\longrightarrow  & H_{k-1}(\w Y) &\oplus & H_{k-1}(\widehat D) &\longrightarrow& H_{k-2}(\widehat E)  \\
\hphantom{q_*}\downarrow{p_* } &&\downarrow{\w q_*} &\searrow & \downarrow{} & &\hphantom{p'_*}\downarrow{{\w q'_*}}  \\
 H_{k} (W ) &\longrightarrow  & H_{k-m}(Y) &\oplus &  H_{k-1}(D)  &\stackrel{ } \longrightarrow &  H_{k-m}(Y)
\end{array}
\end{equation}
where the horizontal arrows are Gysin morphisms. (In particular $H_{k-m}(Y)  \to H_{k-m}(Y)$ is the identity.)  The morphism $\w q_*$, resp. $\w q'_*$, is 
given by Theorem \ref{makingGysinsquare} for the blowup $p:  \w W \to W$, resp. $p': \widehat D \to D$.   Note that the augmented diagram \eqref{keyacyclic2} includes
two acyclic squares of type \eqref{Gysinsquare}.  Taking into account the commutativity of \eqref{keyacyclic1}, in order to  establish the commutativity of \eqref{keyacyclic2} it suffices to show that the square   
\begin{equation}\label{commutative}
\renewcommand{\arraystretch}{1.5}
\begin{array}[c]{ccccc}
   H_{k-1}(\w Y)  &\stackrel{i^*_{\widehat E,\w Y} } \longrightarrow& H_{k-2}(\widehat E)  \\
{\w q_*}  \downarrow{p_{1*}} & &\hphantom{p'_*}\downarrow{{\w q'_*}}  \\
 H_{k-m}(Y) \oplus   H_{k-1}(D)  &\stackrel{\id + i^*_{Y,D}  } \longrightarrow &  H_{k-m}(Y)
\end{array}
\end{equation}
is commutative, which we prove by considering two cases.  

\smallskip	
\noindent \emph {Case 1.} Let $\beta = q^* \alpha \in  H_{k-1} (\w Y)$,  $\alpha \in H_{k-m} (Y )$.   Then $\w q_* \beta = \w q_* q^* \alpha = \alpha$,  
$p_{1*} \beta= (i_{Y,D})_* q_* q^* \alpha =0$, and $\w q'_* i^*_{\widehat E,\w Y} \beta= \w q'_* (q')^* \alpha = \alpha$.  

\smallskip	
\noindent \emph {Case 2.} Let $\beta = s^* (\alpha) \in H_{k-1} (\w Y)$, $\alpha \in H_{k} (\w W )$.
By the commutativity of the left hand  subdiagram of \eqref{keyacyclic2} of type \eqref{Gysinsquare}, $\w q_* \beta=
i^*_{Y,W } p_* \alpha$. Note that both the top and the bottom rows of \eqref{keyacyclic2} are complexes (\emph{i.e.}~the composition of two consecutive morphisms is zero). Indeed, they are the
simple complexes of the Gysin diagrams associated to commutative squares. Hence $i^*_{\widehat E,\w Y} \beta +
i^*_{\widehat E,\widehat D} \w a ^* (\alpha)=0$ and $\w q_* \beta + i^*_{Y,D}
(p_{1*} \beta + p_{2*} \w a ^* (\alpha))=
i^*_{Y,W } p_* \alpha + i^* _{Y,D} i^*_{D,W } p_* \alpha =0$. Therefore we have
$i^*_{\widehat E,\w Y} \beta = i^*_{\widehat E,\widehat D} \w a ^* (\alpha)$ and $\w q_* \beta + i^*_{Y,D} p_{1*} \beta = i^*_{Y,D} p_{2*} \w a ^* (\alpha) $. Now $\w q'_* i^*_{\widehat E,\widehat D} \w a ^* (\alpha)
= i^* _{Y,D} p_{2*} \w a ^* (\alpha) $ by the commutativity of the right hand  subdiagram of \eqref{keyacyclic2} of type \eqref{Gysinsquare}.

By (b) of Lemma \ref{directsum},  $H_{k-1} (\w Y)$ is generated by  $\im q^* $ and $\im s ^* $.  
Thus the commutativity of \eqref{commutative} follows from cases 1 and 2.  

Recall that to say that the diagram \eqref{keyacyclic2} is \emph{acyclic} means that the associated simple complex is acyclic.  
Thus the diagram \eqref{keyacyclic2} is acyclic since it consists of two acyclic squares of type \eqref{Gysinsquare}.  More precisely, the simple complex of \eqref{keyacyclic2}  is acyclic since it equals 
the simple complex of the following diagram with acyclic rows, 
\begin{equation*}
\renewcommand{\arraystretch}{1.5}
\begin{array}[c]{ccccccc}
 H_{k} (\w W) &\longrightarrow  & H_{k-1}(\w Y) &\oplus & H_{k} (W )&\longrightarrow&  H_{k-m}(Y) \\
\hphantom{q_*}\downarrow{ } &&\downarrow{} &\searrow & \downarrow{} & &\hphantom{p'_*}\downarrow{{}}  \\
  H_{k-1}(\widehat D) &\longrightarrow  & H_{k-2}(\widehat E)  &\oplus &  H_{k-1}(D)  &\stackrel{ } \longrightarrow &  H_{k-m}(Y)
\end{array}
\end{equation*}
It follows that \eqref{keyacyclic1}  is acyclic, since the diagrams \eqref{keyacyclic1}  and \eqref{keyacyclic2}  differ by the acyclic diagram 
$ H_{k-m}(Y) \stackrel{\id } \longrightarrow  H_{k-m}(Y)$.  
\end{proof}

Now we prove Theorem \ref{blowup2} by induction on $(r,c)$, where $r= r(D)$ is the number of smooth components $D_i$ of the divisor $D$ such that $Y\subset D_i$ and $c= c(D)$ is the complexity of $D$. (Since $Y$ is irreducible, $r(D)$ equals the number of irreducible components of $D$ that contain $Y$, so $r(D)$ is independent of the nonsingular decomposition $D = \bigcup_i D_i$.) Proposition \ref{keyquasiisomorphism} is the base case $(r,c)=(1,1)$.

Suppose $r(D)=1$ and $c(D)>1$. Let $\mathcal D$ be a nonsingular decomposition of $D$ with $|\mathcal D| = c(D)$. Let $D = D''\cup V$, where $V=D_0\in \mathcal D$ and $Y\not\subset V$, and let $D'=D''\cap V$. We have a diagram with exact rows which is commutative by Corollary \ref{pushforward}:
\begin{equation}\label{exactsequences}
\renewcommand{\arraystretch}{1.5}
\begin{array}[c]{ccccccccc}
 0&\longrightarrow  & G(\w V, \w V\setminus \w D')[1] &\longrightarrow & G(\w W, \w W\setminus \w D)\longrightarrow&  G(\w W, \w W\setminus \w D'')&\longrightarrow & 0 \\
&&\downarrow{b'_*}&&\downarrow{b_*}&\downarrow{b''_*}&&& \\
  0&\longrightarrow  & G( V,  V\setminus  D')[1] &\longrightarrow & G( W,  W\setminus  D)\longrightarrow&  G( W,  W\setminus D'')&\longrightarrow & 0\end{array}
\end{equation}
Now $c(D')< c(D)$ and $c(D'')<c(D)$, so by induction on $c(D)$ the maps $b'_*$ and $b''_*$ are isomorphisms. Therefore $b_*$ is an isomorphism.

Now suppose $r(D)>1$. Let $D = D''\cup V$, where $V=D_0\in \mathcal D$ and $Y\subset V$, and let $D'= D''\cap V$. Then $r(D')< r(D)$ and $r(D'')<r(D)$, so by induction on $r(D)$ the diagram \eqref{exactsequences} shows that $b_*$ is an isomorphism. This completes the proof of Theorem \ref {blowup2}.


\section{Extension of the weight filtration to singular varieties}\label{weightsingularsection}

Following Guill\'en and Navarro Aznar \cite{navarro}, let $\Sch$ be
the category of reduced separated schemes of finite type over $\R$.  In this paper we are interested in the topology of the set of real points of $X\in \Sch$.  
The set $X(\R) $ of real points of $X$ with its sheaf of regular functions  is 
a real algebraic variety in the sense of Bochnak-Coste-Roy \cite{BCR}.

By an \emph{acyclic square} in $\Sch$ we mean a cartesian diagram \begin{equation}\label{acyclic}
\renewcommand{\arraystretch}{1.5}
\begin{array}[c]{ccc}
\w Y&\longrightarrow&\w X\\ 
\downarrow&&\hphantom{p}\downarrow{p} \\ 
Y &\stackrel{\displaystyle i}\longrightarrow&X
\end{array}
\end{equation}
such that $i$ is a closed immersion, $p$ is proper, and $p$ induces an isomorphism $\w X\setminus \w Y \to X\setminus Y$ \cite[(2.1.1)]{navarro}.

Let $\mathbf V(\R)$ be the category of nonsingular projective schemes over $\R$, and let $\mathbf V^2(\R)$ be the category of pairs $(W,X)$ such that $W\in \mathbf V(\R)$, $X$ is an open subscheme of $W$, and $D=W\setminus X$ is a divisor with normal crossings in $W$.

An \emph{elementary acyclic square} in $\mathbf V^2(\R)$ \cite[(2.3.1)]{navarro} is a diagram
\begin{equation}\label{elementaryacyclic}
\renewcommand{\arraystretch}{1.5}
\begin{array}[c]{ccc}
(\w Y,\w Y\cap \w X)&\longrightarrow&(\w W,\w X)\\ 
\downarrow&&\hphantom{p}\downarrow{p} \\ 
(Y,Y\cap X) &\stackrel{\displaystyle i}\longrightarrow&(W,X)
\end{array}
\end{equation}
such that $p:\w W\to W$ is the blowup of $W$ with smooth center $Y$ that has normal crossings with the divisor $D= W\setminus X$, where $\w X = p\inv X$ and $\w Y = p\inv Y$. (The condition that $Y$ has normal crossings with $D$ includes both $Y\not\subset D$ and $Y\subset D$.)

Let $\mathcal C$ be the category of bounded complexes of $\Z_2$ vector spaces with increasing bounded filtration. Following \cite{navarro} we denote by $Ho\, \mathcal C$ the category $\mathcal C$ localized with respect to filtered quasi-isomorphisms. By Proposition $(1.7.5)^{op}$ of \cite{navarro}, the category $\mathcal C$ with this notion of quasi-isomorphism and the simple complex operation for cubical diagrams is a category of homological descent \cite[\S 1A]{weight}.

A $\mathbf \Phi$-\emph{rectification} of a functor $G$ with values in a derived category $\HCat$ is an extension of $G$ to a functor of finite orderable diagrams, with values in the derived category of diagrams, satisfying certain naturality properties \cite[(1.6.5)]{navarro}, \cite[p.~125]{weight}. A factorization of $G$ through the category $\mathcal C$ determines a canonical rectification of $G$.

We define a functor 
$$
\mathbf F:\mathbf V^2(\R)\to Ho\, \mathcal C
$$
as follows. If $(W,X)\in \mathbf V^2(\R)$, then the real algebraic variety $W(\R)$ is a good compactification of $X(\R)$. Let $X'$ be the associated corner compactification of $X(\R)$, and set
$$
\mathbf F(W,X) = \mathcal WC_*(X'),
$$
the weight complex of this good compactification \eqref{weightcomplex}. We have that $\mathbf F$ is a functor by the functoriality of the semialgebraic chain complex $C_*(X')$ \cite[Appendix]{weight} and Theorem \ref{functoriality}.

\begin{thm} \label{navarro}
There is a covariant $\mathbf \Phi$-rectified functor 
$$
\mathcal WC_*:\Sch \to \HCat
$$ 
such that 

(1) if $(W,X)\in \mathbf V^2(\R)$ there is a natural isomorphism $\mathcal WC_*(X) \cong \mathbf F(W,X)$,

(2) $\mathcal WC_*$ satisfies the following acyclicity property:
For an acyclic square (\ref{acyclic})
the simple filtered complex of the diagram
\begin{equation*}
\renewcommand{\arraystretch}{1.5}
\begin{array}[c]{ccc}
\mathcal WC_*(\w Y)&\longrightarrow&\mathcal WC_*(\w X)\\ 
\downarrow&&\downarrow \\ 
\mathcal WC_*(Y) &\longrightarrow&\mathcal WC_*(X)
\end{array}
\end{equation*}
is acyclic (quasi-isomorphic 
to the zero complex).  

Such a functor $\mathcal WC_*$ is unique up to a unique quasi-isomorphism. 
\end{thm}

\begin{proof} 
This theorem follows from applying \cite[Theorem $(2.3.6)^{op}$]{navarro} to the functor $\mathbf F$. 
Since $\mathbf F$ factors through $\mathcal C$, it is automatically $\mathbf\Phi$-rectified \cite[(1.6.5), (1.1.2)]{navarro}. 
Clearly $\mathbf F$   is additive for disjoint unions (condition (2.1.5) (F1) of \cite{navarro}). It remains  to check  condition (2.1.5) (F2) of \cite{navarro}: Given an elementary acyclic square \eqref{elementaryacyclic}, the simple
filtered complex  associated to the  square
\begin{equation*}
\renewcommand{\arraystretch}{1.5}
\begin{array}[c]{ccc}
\mathbf F(\w Y,\w Y\cap \w X)&\longrightarrow&\mathbf F(\w W,\w X)\\ 
\downarrow&&\downarrow \\ 
\mathbf F(Y,Y\cap X) &\longrightarrow&\mathbf F(W,X)
\end{array}
\end{equation*}
is acyclic.  This follows from our blowup results Theorem \ref{blowup1} and Theorem \ref{blowup2}. \end{proof}

\begin{rem}
This theorem shows not only that the weight complex functor extends to singular varieties, but also that, up to quasi-isomorphism, the weight complex  \eqref{weightcomplex} of a smooth real algebraic variety $X$ does not depend on the choice of a good compactification.
\end{rem}

\begin{prop}\label{homology}
For all $X\in\Sch$ the homology of the weight complex $\mathcal WC_*(X)$ is the classical compactly supported homology of $X$ with $\Z_2$ coefficients,
$$
H_*(\mathcal WC_*(X)) = H_*(X).
$$
If $X$ has dimension $n$, for each $k\geq 0$ the filtration of $H_k(X)$ given by this identification satisfies
\begin{equation}\label{singularweight}
0 = \mathcal W_{-n-1}H_k(X)\subset \mathcal W_{-n}H_k(X)\subset\cdots\subset \mathcal W_{0}H_k(X) = H_k(X).
\end{equation}
\end{prop}

\begin{proof} 
The proof of the first assertion is parallel to the proof of  of \cite[Proposition 1.5]{weight}. (One considers the forgetful functor from the category $\mathcal C$ to the category $\mathcal D$ of bounded complexes of $\Z_2$ vector spaces.)  

The second assertion follows from the fact that the weight complex $\mathcal WC_*(X)$ can be computed as the simple filtered complex associated to the diagram of filtered complexes given by a \emph{cubical hyperresolution} of $X$. This is the basic construction of Guill\' en and Navarro Aznar \cite{navarro}. If $\dim X = n$  there is an $n$-cubical diagram $X_\bullet$  in $\Sch$, \emph{i.e.} a contravariant functor from the set of subsets of $\{0,\dots,n\}$ to $\Sch$, with $X= X_\bullet(\emptyset)$, and $X_\bullet(S)$ smooth for $S\neq\emptyset$. For $q\geq 0$, if $X^{(q)}$ is the disjoint union of the smooth schemes $X_\bullet(S)$ for $|S| = q+1$, then $\dim X^{(q)}\leq n-q$, and we have\begin{equation}\label{cubical}
\begin{aligned}
\mathcal W_iC_k(X) &= \bigoplus_{l+q =k} \mathcal W_iC_l(X^{(q)}) ,\\
\partial: \mathcal W_iC_l(X^{(q)})&\to \mathcal W_iC_{l-1}(X^{(q)})\oplus \mathcal W_iC_l(X^{(q-1)}) ,
\end{aligned}
\end{equation}
where $\partial c = \partial'c + \partial''c$, with $\partial'$ the boundary map of the chain complex $C_l(X^{(q)})$ and $\partial''$ the chain homomorphism induced by the map $X^{(q)}\to X^{(q-1)}$ given by the cubical diagram.

By \eqref{weightcomplex} we have $\mathcal W_iC_l(X^{(q)})=0$ for $i<-\dim X^{(q)}= -n+q$ and $\mathcal W_iC_l(X^{(q)})= C_l(X^{(q)})$ for $i\geq -l$. Since $q\geq 0$ and $l\geq 0$, we have $\mathcal W_iC_k(X)= 0$ for $i< -n$ and $\mathcal W_iC_k(X)=C_k(X)$ for $i\geq 0$.
\end{proof}

The filtration \eqref{singularweight} is the \emph{weight filtration} of the homology of $X$. It is an interesting problem to describe the relation of this filtration to Deligne's weight filtration  \cite{deligne} for the complex points of $X$.

If $X\in \Sch$ let $\mathcal WC_*^{BM}(X)$ denote the  weight complex of Borel-Moore chains  (semialgebraic chains with closed supports) of $X(\R)$ defined in \cite[Theorem 1.1]{weight}. 

\begin{prop}\label{oldweight}
There is a natural transformation of functors $\theta: \mathcal WC_*\to \mathcal WC_*^{BM}$.  If $X\in \Sch$ the canonical homomorphism $\varphi_X: H_*(X) \to H_*^{BM}(X) $ is induced by the morphism $\theta_X: \mathcal WC_*(X)\to \mathcal WC_*^{BM}(X)$, and so $\varphi_X$ is compatible with the weight filtrations.  If $X$ is compact (\emph{i.e.} proper over $\R$) the morphism $\theta_X$ is a quasi-isomorphism.
\end{prop}

\begin{proof}
Let $\mathbf {Sch^2_{Comp}}(\R)$ be the category of pairs $(W,X)$, where $W$ is compact and $X$ is an open subscheme of $W$. Theorem (2.3.6) of \cite{navarro} is proved in two steps. The first step is \cite[Theorem (2.3.3)]{navarro}, the extension property for the inclusion $\mathbf V^2(\R)\to \mathbf {Sch^2_{Comp}}(\R)$. By this theorem, our functor $\mathbf F$ on $\mathbf V^2(\R)$ extends to a functor  $\mathbf F'$ on $\mathbf {Sch^2_{Comp}}(\R)$ satisfying the conditions of \cite[Theorem (2.1.5)]{navarro}.  For the second step, the proof of \cite[Theorem (2.3.6)]{navarro} shows that restriction of $\mathbf F'$ to the second factor gives a well-defined functor $\mathcal WC_*$ on $\Sch$ satisfying \cite[(2.1.5)]{navarro}.

Thus we have  a sequence of natural morphisms 
$$
\mathcal WC_*(X)\cong \mathbf F'(\Xb,X) \to \mathbf F'(\Xb,\Xb) \cong \mathcal WC_*(\Xb) \cong \mathcal WC^{BM}_*(\Xb) \to \mathcal WC^{BM}_*(X),
$$
where $\Xb$ is any compactification of $X$, the first and second quasi-isomorphisms are given by the extension results described above, and the third quasi-isomorphism is given by \cite[Corollary (2.3.7)]{navarro}.  If $X$ is compact then we can take $\Xb = X$, in which case the second and fifth morphisms above are identities. 
\end{proof}

Consider the weight filtration \eqref{singularweight} of a variety $X$.
If $X$ is nonsingular and quasi-projective then, by \eqref{weightcomplex}, $\mathcal W_{-k}H_k(X) = H_k(X)$.  If $Y$ is compact then, by Proposition \ref{oldweight} and \cite[p.~129]{weight}, $\mathcal W_{-k-1}H_k(Y)=0$.  Thus if $f: X\to Y$ is a regular morphism from 
a nonsingular quasi-projective variety to a compact variety, then $\im  [f_k :H_k(X) \to H_k(Y)] \subset \mathcal W_{-k}H_k(Y)$ and $\mathcal W_{-k-1}H_k(X) \subset \ker [f_k :H_k(X) \to H_k(Y)]$. 
   Thus if $[c] \in \mathcal W_{-k-1}H_k(X)$ then 
$c$ is a boundary in any algebraic compactification of $X$.   

In the special case of a good compactification this result is sharp:  

\begin{prop}
Let $X$ be a nonsingular quasi-projective variety and let $i : X\to \Xb$ be the inclusion in a good compactification of $X$.  Then for all $k\geq 0$,
$\ker [i_k: H_k(X)\to H_k(\Xb)] = \mathcal W_{-k-1}H_k(X) $.  In particular, $\ker i_k$ does not depend on the choice of a good compactification.  
\end{prop}
\begin{proof}
This follows from Proposition \ref{shortexact}.  
\end{proof}

\begin{example}

Let $X\subset \R^2$ be given by $xy\ne 0$.  The embedding $X\subset  \proj ^2(\R)$ is 
a good compactification of $X$ and  $H_0(X) = (\Z_2)^4$,  $\mathcal W_{-1} H_0(X) = (\Z_2)^3$, and 
$\mathcal W_{-2} H_0(X) = \Z_2$.
Let $Y\subset \R^2$ be given by $x(x-1)(x+1)\ne 0$.  The embedding $Y\subset  \proj ^2(\R)$ is not 
a good compactification since $\proj ^2(\R) \setminus Y$ is the union of four lines intersecting at one point.  
By blowing up this point we obtain a good compactification of $Y$.  Then $H_0(Y) = (\Z_2)^4$, $\mathcal W_{-1} H_0(Y) = (\Z_2)^3$, and
$\mathcal W_{-2} H_0(Y) = 0$.
In particular $X$ and $Y$ are not isomorphic.  
\end{example}

\begin{example}
Let $X =  \proj ^1(\R)\times  \R$.  Then  $\proj ^1(\R)\times  \proj^1(\R )$ is a good compactification of $X$  and  
$\mathcal W_{-2} H_1(X) = 0$.  (The generator of $H_1(X)$ is not a boundary in   $\proj ^1(\R)\times  \proj^1(\R )$.)
Let $Y =  \R^2\setminus 0$.  To obtain a good compactification of $Y$ we embed $Y$ in  $\proj ^2(\R)$  and blow up the origin.   Then $\mathcal W_{-2} H_1(X) = \Z_2$. (The generator of $H_1(Y)$ is already a boundary in   $\proj ^2(\R)$.)
In particular $X$ and $Y$ are not isomorphic.  
\end{example}

\begin{example} 
(1) The inclusion $  \R^* \subset \proj ^1(\R)$ is a good compactification of $ \R^*$.  Thus $\mathcal W_{0} H_0(\R^*) = H_0(\R^*)  =(\Z_2)^2$ and $\mathcal W_{-1} H_0(\R^*)  =\Z_2$.

(2)
Let $X $ be the Bernoulli lemniscate $\{(x^2+y^2) ^2  = x^2-y^2\} \subset \R^2$.  The resolution of $X$ is given by blowing up the origin 
$\pi : \w X \to X$, and then $\w X$ is  diffeomorphic to $S^1$.  The exceptional divisor $E=\pi\inv (0)$ is the union of two points.  Thus  $X_{01}= \{0\}$, $X_{10}= \w X$, $X_{11}= E$ is a cubical hyperresolution of $X$.  
Hence $\mathcal W_{-1} H_1(X) = \Z_2 \subset  H_1(X)=(\Z_2)^2$.   This also follows from \cite[3.3]{weight}.  Indeed, $\mathcal W_{-k} H_k(W) $  is the lowest filtration that can be non-zero for $W$ compact, and  the homology classes 
of $\mathcal W_{-k} H_k(W)$ are precisely those that are represented by the arc-symmetric sets.  
In our case,  the generator of $\mathcal W_{-1} H_1(X)$ is the fundamental class of $X$.  The elements of 
$ H_1(X)\setminus \mathcal W_{-1} H_1(X)$ are represented by the cycles that are halves of the lemniscate; they are 
not arc-symmetric.  

(3)
Let $Y = X\times  \R^* $ where $X$ is  the Bernoulli lemniscate.  Then  $Y_{01} = X_{01}\times  \R^* $, 
 $Y_{10} = X_{10}\times  \R^* $,  $Y_{11} = X_{11}\times  \R^* $ is a cubical hyperresolution of $Y$.  
Hence, using \eqref{cubical}, we obtain that $\Z_2 = \mathcal W_{-2} H_1(Y) \subset \mathcal W_{-1} H_1(Y) =  (\Z_2)^3$, 
$H_1(Y) =(\Z_2)^4$.    The generator of $\mathcal W_{-2} H_1(Y)$ is given by the product 
 of the fundamental class of $X$ with the generator of  $\mathcal W_{-1} H_0(\R^*)$.  
 The generators of  $\mathcal W_{-1} H_1(Y)$ are of two types.  The first type is the  product 
 of the fundamental class of $X$ with a generator of $ H_0(\R^*)$.  The second type is the  product of 
 a generator of $ H_1(X)$ (a  half of the lemniscate) with the generator of $\mathcal W_{-1} H_0(\R^*) $.  
 The sum of these four elements is zero and any three of them generate $\mathcal W_{-1} H_1(Y)$ as a $\Z_2$ vector space.
 Note that the generators of the second type cannot be represented by arc-symmetric cycles.  
\end{example}


\section{Appendix: Discrete torus groups}\label{appendix}

The following discussion is adapted from \cite{BFMH}. Let $(G,\cdot)$ be the group of functions $g: \{1,\dots,n\}\to \{1, -1\}$ with product $(g\cdot h)(k) = g(k)h(k)$. Thus $G =\{1,-1\}^n$,  the set of elements of order 2 of  the torus $(S^1)^n\subset (\C^*)^n$. We refer to $G$ as a \emph{discrete torus} of rank $n$.

Let $(V,+)$ be the additive group corresponding to $(G,\cdot)$. If $g\in G$, let $g'$ denote the corresponding element of $V$, so that $g'+h' = (g\cdot h)'$ and $1' = 0$. Since $g\cdot g= 1$ for all $g\in G$, we have $g' + g' = 0$ for all $g'\in V$, and so $V$ is a vector space over $Z_2$, with $\dim_{\Z_2}V = n$. If $H$ is a subgroup of $G$, we say that $H$ has \emph{rank} $p$ if the corresponding subgroup $H'$ of $V$ has dimension $p$ over $\Z_2$.

Let $A = \Z_2[G]$ be the $\Z_2$ group algebra of $G$. The algebra $A$ is the set of finite formal sums $\sum_i a_i [g_i]$, where $a_i\in \Z_2$ and $g_i\in G$, with addition and multiplication defined by 
$$
\begin{aligned}
\textstyle\sum_i a_i [g_i] + \textstyle\sum_i b_i [g_i] &= \textstyle\sum_i (a_i + b_i)[g_i] ,\\
(\textstyle\sum_i a_i [g_i])(\textstyle\sum_j b_j [g_j]) &= \textstyle\sum_k \sum_{g_ig_j=g_k} (a_ib_j) [g_k].
\end{aligned}
$$
As a vector space over $\Z_2$, the algebra $A$ has dimension $|G| = 2^n$. If  $S$ is a subset of $G$, let $[S] = \sum_{h\in S}[h]\in A$.

Let $\epsilon: A\to \Z_2$ be the \emph{augmentation map},
\begin{equation}\label{augmentation}
\epsilon(\textstyle\sum_i a_i [g_i]) = \textstyle\sum_i a_i,
\end{equation}
and let $\I = \Ker\epsilon$ be the \emph{augmentation ideal}. Consider the filtration of the algebra $A$ by the ideals $\I^p$ for $p\geq 1$,
\begin{equation}\label{Ifiltration}
A\supset \I^1\supset\I^2\supset\I^3\supset\cdots.
\end{equation}

\begin{lem}\label{rank} For each $p\geq 1$ the ideal $\I^p$ is spanned as a vector space by the elements $[H]$ such that $H$ is a subgroup of $G$ and $\operatorname{rank} H = p$. 
\end{lem} 

\begin{proof} 
We proceed by induction on $p$. For $p=1$, we have $\alpha \in \I$ if and only if $\alpha = \sum_{g \in S} [g]$, where $|S|$ is even. Then $\alpha = \sum_{1\neq g\in S} ([1] + [g])$, and $[1]+[g] = [\{1,g\}]$, with $\operatorname{rank}\{1, g\} = 1$.

Now suppose $\I^p$ is spanned by the elements $[H]$ with $\operatorname{rank} H = p$. Then $\I^{p+1}$ is spanned by elements of the form $([1] + [g])[H]$. If $g\in H$ then $([1] + [g])[H] = [H]+[H] = 0$. If $g\notin H$ then $([1]+[g])[H] = [K]$, where $K$ is the subgroup of  rank $p+1$ generated by $H$ and $g$.
\end{proof}

\begin{prop}\label{gr} There is a canonical  isomorphism $\Phi:\Lambda^* V\stackrel{\displaystyle \approx}\longrightarrow \operatorname{Gr}_\I A $ of graded algebras which induces vector space isomorphisms $\Lambda^p V \cong \I^p/\I^{p+1}  $ for each $p\geq 1$. Moreover, $\Phi$ is an isomorphism of functors; \emph{i.e.} $\Phi$ is functorial with respect to homomorphisms of the group $G$.
\end{prop}

\begin{proof}
We claim that the function $\phi: V \to \I$ given by $\phi(g') = [1]+[g]$ induces a homomorphism of graded algebras
$$
\Phi: \Lambda^*V \to \operatorname{Gr}_\I A,
$$
with $\Phi(\Lambda^pV) = \I^p/\I^{p+1}$.
We have
$$
\begin{aligned}
\phi(g'+h') &= \phi((gh)') = [1]+ [gh], \\
\phi(g') + \phi(h') &= ([1]+[g])+([1]+[h]) = [g] + [h].
\end{aligned}
$$
Now 
$$
([1] + [gh]) + ([g]+[h]) = ([1]+[g])([1]+[h]) \in \I^2,
$$
so $\phi$ defines an additive homomorphism $V\to \I/\I^2$. Thus the function
$$
\begin{aligned}
\phi_p:\otimes^pV &\to \I^p/\I^{p+1},\\
\phi_p(g_1\otimes\cdots\otimes g_p) &= ([1] + [g_1])\cdots([1] + [g_p]),
\end{aligned}
$$
is multilinear.
Since $([1]+[g])^2 = [1] + 2[g] + [g^2] = 0$ for all $g\in G$, the maps $\phi_p$ define an algebra homomorphism $\Phi: \Lambda^*V \to \operatorname{Gr}_\I A$. If $(g_1',\dots,g_p')$ is a basis for the subspace $H'\subset V$ corresponding to  the subgroup $H\subset G$, then  $([1] + [g_1])\cdots([1]+[g_p]) = [H]$. Thus $\Phi$ is surjective by Lemma \ref{rank}. Since $\dim_{\Z_2} \Lambda^*V = 2^n = \dim_{\Z_2}   \operatorname{Gr}_\I A$, we conclude that $\Phi$ is an isomorphism. 

If $\gamma: G\to H$ is a homomorphism of discrete torus groups, the commutativity of the diagram
\begin{equation*}
\renewcommand{\arraystretch}{1.5} 
\begin{array}[c]{ccc}
V_G&\stackrel{\displaystyle \phi_G}\longrightarrow& \I_G\\ 
\hphantom{\gamma_*}\downarrow{\gamma_*}&&\hphantom{\gamma_*}\downarrow{\gamma_*} \\ 
V_H&\stackrel{\displaystyle \phi_H}\longrightarrow&\I_H
\end{array}
\end{equation*}
implies that $\Phi$ is functorial.
\end{proof}

For $J\subset \{1,\dots,n\}$ let $G(J)=\{g\in G\ |\ g(i) = 1,\ i\notin J\}$. 

\begin{cor}\label{nchoosep} 
For each $p\geq 1$ we have $\dim_{\Z_2} \I^p / \I^{p+1} = \binom {n}{p}$. In particular $\I^{n+1} = 0$.
The subset $\{[G(J)]\ |\ |J| = p\}$ of $\I^p$ maps to a basis of\, $\I^p / \I^{p+1}$.
\end{cor}

For $n,m\in \N$, let  $(G_n,\cdot)$ be the group of functions $g: \{1,\dots,n\}\to \{1, -1\}$, and let $ (G_m,\cdot)$ be the group of functions $g: \{1,\dots,m\}\to \{1, -1\}$.  Let  $\gamma:  G_n \to G_m$ be a group homomorphism.  Then $\gamma$ is given by an $m\times n$ matrix $(a_{ij})$ with coefficients in $\Z_2$,
 \begin{equation}\label{matrix}
 \gamma (g_1, \ldots , g_n) = \left (\prod g_i^{a_{i1}}, \ldots , \prod g_i^{a_{im}}\right ).
 \end{equation}
 Let $A(G_n)$ and $A(G_m)$ be the group algebras of $G_n$ and $G_m$, respectively, and let $\I^p (G_n)$ and  $\I^p (G_m)$ be the associated filtrations \eqref{Ifiltration}. The group homomorphism $\gamma$ induces an algebra homomorphism $\gamma_*: A(G_n)\to A(G_m)$ with $\gamma_* (\I(G_n)) \subset \I (G_m)$, and so for all $p \ge 1$ we have $\gamma_* (\I^p(G_n)) \subset \I^p (G_m)$. 

 \begin{cor}\label{morphism} If $\gamma:G_n\to G_m$ is a group homomorphism, the induced linear map 
 $\gamma_*:\I^p (G_n)/ \I^{p+1}(G_n) \to \I^p (G_m)/ \I^{p+1}(G_m) $ is given by a matrix $(a_{IJ})$ with respect to the bases of Corollary \ref{nchoosep}, where 
 for $J\subset \{1, \ldots, n\}$, $I\subset \{1,\ldots , m\}$, $|I|=|J|=p$, 
 $$
a_{I J} = \det (a_{ij}) _{i\in I, j\in J}\ .
$$
 \end{cor}

\begin{prop}\label{generating}  
For all $p\geq 1$ the vector subspace $\I^p\subset A$ is spanned by the set of translates $\{ [gG(J)]\ |\ g\in G, \ |J| = p\}$.     
\end{prop}

\begin{proof}
Let $F^p$ be the subspace of $A$ spanned by $\{[gG(J)]\ |\ g\in G,\ |J|=p\}$. Since $\I^p$ is an ideal of $A$ we have $F^p\subset \I^p$.

Let $E^p\subset F^p$ be the subspace spanned by $\{[G(J)]\ |\ |J|=p\}$. If $|J|=p$ and $|J'|=p+1$, with $J'= J\cup\{i\}$, then $[G(J')] = [G(J)] + [g_iG(J)]$. Thus $E^{p+1}\subset F^p$, and so $F^{p+1}\subset F^p$. Therefore $F^l\subset F^p$ for all $l\geq p$.

It follows that $\I^p\subset F^p$. For if $\alpha\in \I^p$ then by Corollary \ref{nchoosep} we have 
$$
\alpha = \alpha_p + \alpha_{p+1} +\cdots + \alpha_n,\ \ \alpha_l\in E^l\ .
$$ 
Thus for all $l$ we have $\alpha_l\in F^p$, and so $\alpha\in F^p$.
\end{proof}

\begin{prop}\label{translation}
If $g\in G$ let $\psi_g: A\to A$ be the linear isomorphism of $A$ given by translation by $g$, $\psi_g(\alpha) = [g]\cdot\alpha$. For all $p\geq 1$, $\psi_g(\I^p)= \I^p$, and $\psi_g$ induces the identity map on $\I^p/\I^{p+1}$.
\end{prop}

\begin{proof}
We have $\psi_g(\I^p)= \I^p$ by Proposition \ref{generating}.
Let $|J]=p$. If $g\in G(J)$ then $\psi_g[G(J)]  = [G(J)]$. If $g\notin G(J)$ then 
$$
[G(J)]- \psi_g[G(J)]  = ([1]-[g])[G(J)]\in \I^{p+1}.
$$
Thus $\psi_g$ induces the identity map on $\I^p/\I^{p+1}$ by Corollary \ref{nchoosep}.
\end{proof}


\end{document}